\documentclass[reqno,12pt]{amsart}
\theoremstyle{plain}    
\newtheorem{thm}{Theorem}[section]
\newtheorem{cor}[thm]{Corollary} 
\newtheorem{lemma}[thm]{Lemma} 
\newtheorem{prop}[thm]{Proposition}
\newtheorem{case}{Case}[thm]

\newtheorem{thmsubs}{Theorem}[subsection]
\newtheorem{lemmasubs}[thmsubs]{Lemma} 
\newtheorem{propsubs}[thmsubs]{Proposition}

\theoremstyle{remark}

\newtheorem{remark}[thm]{Remark}

\theoremstyle{definition}

\newtheorem{defi}[thm]{Definition}
\newtheorem{example}[thm]{Example}

\usepackage{amscd,amssymb,comment,eepic,euscript,graphics}

\newcount\theTime
\newcount\theHour
\newcount\theMinute
\newcount\theMinuteTens
\newcount\theScratch
\theTime=\number\time
\theHour=\theTime
\divide\theHour by 60
\theScratch=\theHour
\multiply\theScratch by 60
\theMinute=\theTime
\advance\theMinute by -\theScratch
\theMinuteTens=\theMinute
\divide\theMinuteTens by 10
\theScratch=\theMinuteTens
\multiply\theScratch by 10
\advance\theMinute by -\theScratch
\def\timeHHMM{{\number\theHour:\number\theMinuteTens\number\theMinute}}

\def\today{{\number\day\space
 \ifcase\month\or
  January\or February\or March\or April\or May\or June\or
  July\or August\or September\or October\or November\or December\fi
 \space\number\year}}

\def\timeanddate{{\timeHHMM\space o'clock, \today}}

\newcommand\Ac{{\mathcal{A}}}

\newcommand\Afr{{\mathfrak A}}

\newcommand\Cpx{{\mathbf C}}

\newcommand\dom{\operatorname{dom}}

\newcommand\domproj{\operatorname{domproj}}

\newcommand\eps{\epsilon}

\newcommand\et{{\tilde e}}

\newcommand\ft{{\tilde f}}

\newcommand\gt{{\tilde g}}

\newcommand\Hom{\mathrm{Hom}}

\newcommand\HEu{{\EuScript H}}                   

\newcommand\It{{\widetilde I}}

\newcommand\itil{{\tilde i}}

\newcommand\jt{{\tilde j}}

\newcommand\Jt{{\widetilde J}}

\newcommand\kerproj{\operatorname{kerproj}}

\newcommand\kt{{\tilde k}}

\newcommand\Kt{{\widetilde K}}

\newcommand\lambdat{{\tilde\lambda}}

\newcommand\Mcal{{\mathcal{M}}} 
\newcommand\mut{{\tilde{\mu}}}

\newcommand\Nats{{\mathbf N}}

\newcommand\nut{{\tilde{\nu}}}

\newcommand\Proj{{\mathrm{Proj}}}

\newcommand\pt{{\tilde p}}

\newcommand\qt{{\tilde q}}

\newcommand\ran{\operatorname{ran}}

\newcommand\ranproj{\operatorname{ranproj}}

\newcommand\Reals{{\mathbf R}}

\newcommand\smd[2]{\underset{#2}{#1}}

\newcommand\smdp[3]{\overset{#3}{\smd{#1}{#2}}}

\newcommand\St{{\widetilde S}}

\newcommand\Tt{{\widetilde T}}

\newcommand\Ut{{\widetilde U}}

\begin{document}

\title[Horn inequalities]{On a reduction
procedure for Horn inequalities in finite von~Neumann algebras}

\author[Collins]{Beno\^\i{}t Collins$^{\dagger}$}
\address{B.\ Collins,
Department of Mathematics and Statistics, University of Ottawa,
585 King Edward,
Ottawa, ON
K1N 6N5 Canada, and
CNRS, Department of Mathematics, Lyon 1 Claude Bernard University} 
\email{bcollins@uottawa.ca}
\thanks{\footnotesize $^{\dagger}$Research supported in part by NSERC
grant RGPIN/341303-2007}

\author[Dykema]{Ken Dykema$^{*}$}
\address{K.\ Dykema, Department of Mathematics, Texas A\&M University,
College Station, TX 77843-3368, USA}
\email{kdykema@math.tamu.edu}
\thanks{\footnotesize $^{*}$Research supported in part by NSF grant DMS-0600814}



\date{\timeanddate}

\begin{abstract}
We consider the analogues of the Horn inequalities in finite von Neumann algebras,
which concern the possible spectral distributions of sums $a+b$ of self--adjoint elements
$a$ and $b$ in a finite von Neumann algebra.
It is an open question whether all of these Horn inequalities must hold in all finite von Neumann
algebras, and this is related to Connes' embedding problem.
For each choice of integers $1\le r\le n$, there is a set $T^n_r$ of Horn triples 
$(I,J,K)$ of $r$--tuples of integers, and the Horn inequalities are in one--to--one correspondence
with $\cup_{1\le r\le n}T^n_r$.
We consider a property P$_n$, analogous to one introduced by Therianos and Thompson in the case of matrices,
amounting to the existence of projections having certain properties relative to arbitrary flags,
which guarantees that a given Horn inequality holds in all finite von Neumann algebras.
It is an open question whether all Horn triples in $T^n_r$ have property P$_n$.
Certain triples in $T^n_r$ can be reduced to triples in $T^{n-1}_r$ by
an operation we call {\em TT--reduction}.
We show that property P$_n$ holds for the original triple if property P$_{n-1}$ holds
for the reduced one.
A major part of this paper is devoted to showing that this operation of reduction preserves the
value of the corresponding Littlewood--Richardson coefficients.
We then characterize the TT--irreducible Horn triples in $T^n_3$, for arbitrary $n$,
and for those LR--minimal ones
(namely, those having Littlewood--Richardson coefficient equal to $1$), we
perform a construction of projections with respect to flags in arbitrary von Neumann algebras
in order to prove property P$_n$ for them.
This shows that all LR--minimal triples in $\cup_{n\ge3}T^n_3$ have property P$_n$, and so that
the corresponding Horn inequalities hold in all finite von Neumann algebras.
\end{abstract}

\maketitle

\section{Introduction and description of results}

If $A$ and $B$ are Hermition $n\times n$ matrices whose eigenvalues
(repeated according to multiplicity)
are $\alpha_1\ge\alpha_2\ge\cdots\ge\alpha_n$
and $\beta_1\ge\beta_2\ge\cdots\beta_n$, respectively, what can the eigenvalues of $A+B$ be?
In~\cite{H62}, A.\ Horn described sets $T^n_r$ of triples $(I,J,K)$ of subsets of $\{1,\ldots,n\}$,
with $|I|=|J|=|K|=r$,
and conjectured that a weakly decreasing real sequence $\gamma_1\ge\gamma_2\ge\cdots\ge\gamma_n$
can arise as the eigenvalues of $A+B$, for some $A$ and $B$ as above, if and only if
\begin{equation}
\sum_{i=1}^n\alpha_i+\sum_{j=1}^n\beta_j=\sum_{k=1}^n\gamma_k
\end{equation}
and for each triple $(I,J,K)\in\bigcup_{r=1}^{n-1}T^n_r$, the so--called {\em Horn inequality}
\begin{equation}\label{eq:HIJK}
\sum_{i\in I}\alpha_i+\sum_{j\in J}\beta_j\ge\sum_{k\in K}\gamma_k
\end{equation}
holds.
(We recall Horn's definition of the sets $T^n_r$ in Section~\ref{sec:Horn}: see~\eqref{eq:Unr} and~\eqref{eq:Tnr}.)
Horn's conjecture has been proved, due to work of Klyatchko, Tataro, Knutson and Tao.
See the article~\cite{F00} of Fulton.

The purpose of this paper is to prove that analogues of some of the Horn inequalities hold in all finite
von Neumann algebras.
This question was first considered by Bercovici and Li in~\cite{BL01} (see also~\cite{BL06})
and the following exposition is essentially from their papers.
Let $\Mcal$ be a von Neumann algebra with a fixed normal, faithful, tracial state $\tau$.
If $a=a^*\in\Mcal$, the {\em eigenvalue function} of $a$ is the non--increasing, right--continuous function
$\lambda_a:[0,1)\to\Reals$ given by
\begin{equation}
\lambda_a(t)=\sup\{x\in\mathbb{R}\mid\mu_a((x,\infty))>t\},
\end{equation}
where $\mu_a$ is the distrubtion of $a$, which is the Borel measure supported on the spectrum of $a$
and satisfying
\begin{equation}
\tau(a^k)=\int_{\mathbb{R}}t^k\,d\mu_a(t)\qquad(k\ge1).
\end{equation}
For example, if 
\begin{equation}\label{eq:ainMn}
a=a^*\in M_n(\Cpx)\hookrightarrow\Mcal
\end{equation}
has eigenvalues $\alpha_1\ge\alpha_2\ge\cdots\ge\alpha_n$, then
\begin{equation}
\lambda_a(t)=\alpha_j\,,\quad \frac{j-1}n\le t<\frac jn\,,\quad j\in\{1,\ldots,n\}.
\end{equation}

\begin{defi}
Let $(I,J,K)\in T^n_r$ be a Horn triple.
We say that {\em the Horn inequality corresponding to $(I,J,K)$ holds} in $(\Mcal,\tau)$ if
\begin{equation}\label{eq:contsHorn}
\int_{\omega_I}\lambda_a(t)\,dt+\int_{\omega_J}\lambda_b(t)\,dt\ge\int_{\omega_K}\lambda_{a+b}(t)\,dt
\end{equation}
for all $a,b\in\Mcal_{s.a.}:=\{x\in\Mcal\mid x=x^*\}$, where
\begin{equation}
\omega_I=\bigcup_{i\in I}\bigg[\frac{i-1}n,\frac in\bigg)
\end{equation}
and similarly for $\omega_J$ and $\omega_K$.
\end{defi}

Note that~\eqref{eq:contsHorn}
becomes the usual Horn inequality~\eqref{eq:HIJK} when $a$ and $b$ lie in the same copy of the $n\times n$ matrices,
as in~\eqref{eq:ainMn}.

Bercovici and Li showed in~\cite{BL01} that the Horn inequalities corresponding to the Freede--Thompson inequalities
(and certain generalizations of them)
hold in all finite von Neumann algebras.
In~\cite{BL06}, they showed that if $(\Mcal,\tau)$ satisfies Connes' embedding property,
namely, if it embeds
in the ultraproduct $R^\omega$ of the hyperfinite II$_1$--factor, or
equivalently (assuming separable pre--dual), if all $n$--tuples 
of self--adjoints in $\Mcal$ can be approximated in mixed moments by matrices, then all Horn inequalities
hold in $(\Mcal,\tau)$.
Moreover, they showed that the set of possible triples $(\lambda_a,\lambda_b,\lambda_{a+b})$
for $a$ and $b$ self--adjoints
in $R^\omega$ is characterized by the inequalities of the form~\eqref{eq:contsHorn}.
It is an important open question, known as {\em Connes' embedding problem,}
whether all finite von Neumann algebras having separable pre--dual satisfy Connes' embedding property.
In the converse direction, in~\cite{CD} we showed that if certain versions of the Horn inequalities
with matrix coefficients hold in all finite von Neumann algebras, then Connes' embedding problem has a
positive answer.
Seen in this light, it is quite interesting to learn about which Horn inequalities must hold in all finite von Neumann
algebras.
Some speculative observations about possible constructions of counter--examples to embeddability
are found in Section~\ref{sec:constr}.

One method of proving that the Horn inequality corresponding to a given
Horn triple $(I,J,K)\in T^n_r$ holds in a finite von Neumann algebra $(\Mcal,\tau)$
is to construct projections in $\Mcal$ satisfying certain properties with respect to
flags of projections in $(\Mcal,\tau)$.
We say $(I,J,K)$ has {\em property P$_n$} if such projections can always be constructed,
and we introduce a weaker, approximate version of this property.
See the first part of Section~\ref{sec:Horn} for details,
but note that Definition~\ref{def:Pn} and Proposition~\ref{prop:flagHorn} are
for the symmetric reformulation of the Horn sets described there.
Bercovici and Li's proof~\cite{BL01} that certain Horn inequalities must hold
in all finite von Neumann algebras was, to rephrase it, made by showing that they have property P$_n$.
Following their proof we prove the slightly stronger statement that this is implied by property AP$_n$.

In~\cite{TT74}, Therianos and Thompson proved a reduction result,
showing that the analogue of property P$_n$ in $n\times n$ matrices for a given triple $(I,J,K)$ can sometimes be
deduced from the same analogue of property P$_{n-1}$ for a related triple $(\It,\Jt,\Kt)$.
(See also~\cite{Z66}.)
They then used this reduction result and some explicit constructions of projections in matrices
to show that Horn inequalities in $M_n(\Cpx)$ corresponding to triples in $T^n_3$ hold for all $n$.
We show (Lemma~\ref{lem:TT1}) that a similar reduction technique holds for properties
P$_n$ and AP$_n$ in finite von Neumann algebras.
Using this reduction result, though we were not able to prove that Horn inequalities in finite von Neumann
algebras hold for all triples in $\bigcup_{n\ge3}T^n_3$, we do show that they hold for all
the LR--minimal triples in this set.
The moniker LR--minimal refers to the Littlewood--Richardson coefficient of the triple (see
Definition~\ref{def:LRc} and Lemma~\ref{lem:cIJK}); it follows from Theorem~13 of~\cite{F00}
that the set of Horn inequalities coming from LR--minimal triples determines the remaining Horn inequalities,
both in the case of matrices and of finite von Neumann algebras.
As a byproduct of our reduction technique,
we also show that all the Horn inequalities corresponding to triples in
$\bigcup_{r\in\{1,2\},\,n\ge r}T^n_r$ hold in all finite von Neumann algebras,
though this is can be more easily proved directly.
As perhaps the most arduous part of our proof, we show
(Proposition~\ref{prop:redLR}) that the reduction method refered to above preserves the
Littlewood--Richardson coefficient.

Here is a brief description of the rest of this paper.
In Section~\ref{sec:prelims}, we cover some preliminary and (mostly) well known facts
about finite von Neumann algebras.
In Section~\ref{sec:Horn}, we first describe minor reformulation of Horn's triples;
the reformulated set is denoted $\Tt^n_r$, and is invariant under  the obvious action
of the group of permuatations of three letters.
Then we prove the analogue in finite von Neumann algebras of the reduction result
from~\cite{TT74}.
Triples that cannot be reduced are called irreducible, naturally enough.
After introducing new notation $c^{(n)}(I,J,K)$ for Littlewood--Richardson coefficient of $(I,J,K)\in\Tt^n_r$
and observing the
invariance of this quantity under permuting the arguments $I$, $J$ and $K$,
we prove that it is also invariant under the reduction method referred to above.
We then characterize the irreducible triples in $\Tt^n_3$,
compute their Littlewood--Richardson coefficients, and list the irreducible triples of
minimal Littlewood--Richardson coefficient in $\Tt^n_4$, for $n\le9$.
In Section~\ref{sec:Constr}, we exhibit a construction of projections in finite von Neumann algebras
that suffices to prove that the Horn inequalities for all LR--minimal triples in $\bigcup_{n\ge3}T^n_3$
hold in all finite von Neumann algebras.
Merely because we like the argument involving almost invariant subspaces, we prove that property AP$_6$
holds for a certain
element of $\Tt^6_3$ having Littlewood--Richardson coefficient equal to $2$.
Section~\ref{sec:constr}, which is independent of the rest of the paper and can safely be skipped,
contains some speculative remarks about how one might construct a non--embeddable finite von Neumann algebra.

\section{Preliminaries concerning finite von Neumann algebras}
\label{sec:prelims}

In the following three subsections, we review some facts, introduce
some notation and state some results that will be used later.
While (most of) these are certainly not original,
for convenience, we provide some proofs.

\subsection{Two projections}
\label{subsec:twoproj}

Let $\Mcal\subseteq B(\HEu)$ be a finite von Neumann algebra with a fixed faithful, tracial state $\tau$.
Let $\Proj(\Mcal)$ denote the set of self--adjoint idempotents in $\Mcal$, which are also called
{\em projections} in $\Mcal$.
Many elementary but useful facts about projections in $\Mcal$ follow from the standard description of 
the subalgebra generated by any two of them, which we now describe.
Let $p,q\in\Proj(\Mcal)$.
Recall that $p\wedge q$ denotes the projection onto the closed subspace $p\HEu\cap q\HEu$
and $p\vee q$ denotes the projection onto the closure of $p\HEu+q\HEu$.
Let $\Ac=W^*(\{p,q,1\})$ be unital von Neumann algebra generated by $p$ and $q$.
Let $\Afr$ denote the universal, unital C$^*$--algebra generated by two projections $P$ and $Q$.
As is well--known,
\begin{equation}
\Afr\cong\{f:[0,1]\to M_2(\Cpx)\mid f\text{ continuous, }f(0),\,f(1)\text{ diagonal}\},
\end{equation}
with
\begin{equation}
P=\begin{pmatrix}1&0\\0&0\end{pmatrix},\qquad
Q=\begin{pmatrix}t&\sqrt{t(1-t)}\\\sqrt{t(1-t)}&1-t\end{pmatrix}.
\end{equation}
We have a quotient map $\pi:\Afr\to\Ac$ sending $P$ to $p$ and $Q$ to $q$, and
$\Ac$ is isomorphic to the weak closure of the image of the Gelfand--Naimark--Segal
representation of $\Afr$ arising from the trace $\tau\circ\pi$ on $\Afr$.
We thereby identify $\Ac$ with
\begin{equation}\label{eq:vNpq}
\smdp{\Cpx}{\gamma_{11}}{p\wedge q}
  \oplus\smdp{\Cpx}{\gamma_{10}}{p\wedge(1-q)}
  \oplus L^\infty(\mu)\otimes M_2(\Cpx)
  \oplus\smdp{\Cpx}{\gamma_{01}}{(1-p)\wedge q}
  \oplus\smdp{\Cpx}{\gamma_{00}}{(1-p)\wedge(1-q)},
\end{equation}
where $\gamma_{ij}\ge0$, where $\mu$ is a measure concentrated on a subset of the open interval $(0,1)$,
and where the notation in~\eqref{eq:vNpq} means, for example, that $p\wedge q$
is the projection 
\begin{equation}
p\wedge q=1\oplus0\oplus\left(\begin{smallmatrix}0&0\\0&0\end{smallmatrix}\right)\oplus0\oplus0
\end{equation}
and $\tau(p\wedge q)=\gamma_{11}$.
We have
\begin{align}
p&=1\oplus 1\oplus\begin{pmatrix}1&0\\0&0\end{pmatrix}\oplus0\oplus0 \label{eq:p} \\
q&=1\oplus 0\oplus\begin{pmatrix}t&\sqrt{t(1-t)}\\\sqrt{t(1-t)}&1-t\end{pmatrix}\oplus1\oplus0
  \label{eq:q}
\end{align}
and, if
\begin{equation}
a=\lambda_{11}\oplus\lambda_{10}\oplus\begin{pmatrix}f_{11}&f_{12}\\f_{21}&f_{22}\end{pmatrix}
  \oplus\lambda_{01}\oplus\lambda_{00}\in\Ac
\end{equation}
for $\lambda_{ij}\in\Cpx$ and $f_{pq}\in L^\infty(\mu)$, then
\begin{equation}
\tau(a)=\lambda_{11}\gamma_{11}+\lambda_{10}\gamma_{10}
 +\frac12\int(f_{11}+f_{22})\,d\mu
 +\lambda_{01}\gamma_{01}+\lambda_{00}\gamma_{00}.
\end{equation}
Thus, the total mass of $\mu$ is
\begin{equation}
|\mu|=1-\gamma_{11}-\gamma_{10}-\gamma_{01}-\gamma_{00}.
\end{equation}
Of course, if in~\eqref{eq:vNpq} some $\gamma_{ij}$ or $\mu$ itself should be zero,
then the corresponding summand should be understood to be absent.
Inspecting this situation, we observe the following elementary result.
\begin{propsubs}\label{prop:taup}
We have
\begin{gather}
p\vee q=1-(1-p)\wedge(1-q), \\
\tau(p\vee q)=\tau(p)+\tau(q)-\tau(p\wedge q) \label{eq:tpvq} \\
\tau(p-(1-q)\wedge p)=\tau(q-(1-p)\wedge q). \label{eq:p1qp}
\end{gather}
\end{propsubs}

And the following useful lemmas are also immediate.
\begin{lemmasubs}
Then there is a projection $r\in\Ac$ such that $q\le r+p$ 
and $r$ is unitarily equivalent in $\Ac$ to $q-q\wedge p$.
In particular, we have $\tau(r)=\tau(q)-\tau(q\wedge p)$.
\end{lemmasubs}
\begin{proof}
We let
\begin{equation}
r=0\oplus 0\oplus\begin{pmatrix}0&0\\0&1\end{pmatrix}\oplus1\oplus0.
\end{equation}
\end{proof}

\begin{lemmasubs}\label{lem:pqH}
The projection onto the closure of $pq\HEu$ is equal to
$p-p\wedge(1-q)$.
\end{lemmasubs}
\begin{proof}
Multiply the right--hand--sides of~\eqref{eq:p} and~\eqref{eq:q}.
\end{proof}

\subsection{Affiliated operators}

One of the virtues of a finite von Neumann algebra is that its set of affiliated operators forms
an algebra.
Here we briefly review this situation.
Recall that a closed, densely defined, (possibly unbounded) operator $X$ from $\HEu$ to itself
is said to be
{\em affiliated} with $\Mcal$ if, letting $X=v|X|$ be the polar decomposition of $X$,
we have $v\in\Mcal$ and all spectral projections of the positive operator $|X|$ lie in $\Mcal$.
Thus, we have
\begin{equation}\label{eq:|X|int}
|X|=\int_{[0,\infty)}tE_{|X|}(dt),
\end{equation}
for a projection--valued measure $E$, taking Borel subsets of $[0,\infty)$ to elements of $\Proj(\Mcal)$.
Since $\lim_{K\to+\infty}\tau(E_{|X|}([0,K])=1$, we easily see that, if $p\in\Proj(\Mcal)$, then
$p\HEu\cap\dom(X)$ is dense in $p\HEu$, where $\dom(X)$ denotes the domain of $X$.
Thus, we see that if $S$ and $T$ are densely defined operators affiliated with $\Mcal$,
then $S+T$ and $ST$ are densely defined and affiliated with $\Mcal$.

We now define some terms and notation and make some observations
that we will need later.
Let $X$ be a closed,
densely defined operator from $\HEu$ to itself, having polar decomposition $X=v|X|=|X^*|v$
and where $E_{|X|}$ is the spectral measure of the positive operator $|X|$.
The {\em kernel projection} $\kerproj(X)$ of $X$ is the
projection onto $\ker(X)$, and
the {\em domain projection} of $X$ is $\domproj(X)=1-\kerproj(X)$.
Thus,
\begin{equation}
\begin{aligned}
\kerproj(X)&=E_{|X|}(\{0\}) \\
\domproj(X)&=E_{|X|}((0,+\infty))=v^*v
\end{aligned}
\end{equation}
and
\begin{equation}
X=X\cdot\domproj(X).
\end{equation}
The {\em range projection} of $X$ is $\ranproj(X)\in\Proj(\Mcal)$ that is the projection onto
the closure of the range of $X$.
Thus,
\begin{equation}
\ranproj(X)=E_{|X^*|}((0,+\infty))=vv^*
\end{equation}
and
\begin{equation}
X=\ranproj(X)\cdot X.
\end{equation}
Therefore, we have
\begin{equation}\label{eq:taudomran}
\tau(\domproj(X))=\tau(\ranproj(X)).
\end{equation}
The {\em partial inverse} of $X$ is the operator
$Y=|Y^*|v^*$ where
\begin{equation}
|Y^*|=\int_{(0,\infty)}t^{-1}E_{|X|}(dt).
\end{equation}
Thus,
\begin{align}
XY&=\ranproj(X)=\domproj(Y) \\
YX&=\domproj(X)=\ranproj(Y).
\end{align}
Indeed, the restriction of $X$
is an injective linear operator from $\domproj(X)\HEu\cap\dom(X)$ onto $\ran(X)$, and
the restriction of $Y$ to $\ran(X)$ is this operator's inverse.

Let
\begin{equation}
X^\sharp:\Proj(\Mcal)\to
         \{q\in\Proj(\Mcal)\mid q\le\ranproj(X)\}
\end{equation}
be the map defined by
\begin{equation}
X^\sharp(p)=\ranproj(Xp).
\end{equation}
Clearly, $X^\sharp$ is order preserving
and, moreover, if $X$ and $Z$ are operators affiliated with $\Mcal$, then for any $p\in\Proj(\Mcal)$,
\begin{equation}\label{eq:XZp}
(XZ)^\sharp(p)=\ranproj(XZp)=\ranproj(X(Z^\sharp(p)))=X^\sharp Z^\sharp(p).
\end{equation}
\begin{lemmasubs}
Restricting $X^\sharp$ gives a bijection
\begin{equation}\label{eq:Xshbij}
\{p\in\Proj(\Mcal)\mid p\le\domproj(X)\}\to\{q\in\Proj(\Mcal)\mid q\le\ranproj(X)\}.
\end{equation}
Moreover, this bijection is trace preserving and a lattice isomorphism.
Finally, for any $p\in\Proj(\Mcal)$, we have
\begin{align}
X^\sharp(p)&=X^\sharp\big(\domproj(X)-(1-p)\wedge\domproj(X)\big)  \label{eq:Xshp} \\
\tau(X^\sharp(p))&=\tau(p)-\tau(p\wedge\kerproj(X)). \label{eq:tauXshp}
\end{align}
\end{lemmasubs}
\begin{proof}
Clearly, the restriction of $X^\sharp$ provides an order preserving map~\eqref{eq:Xshbij}.
Let $Y$ be the partial inverse of $X$, described above.
If $p\in\Proj(\Mcal)$ and $p\le\domproj(X)$,
then $YXp=p$ and consquently, using~\eqref{eq:XZp}, we have
$Y^\sharp X^\sharp(p)=p$.
Similarly, if $q\in\Proj(\Mcal)$ and $q\le\ranproj(X)$, then $XY=q$ and, consequently,
$X^\sharp Y^\sharp(q)=q$.
This shows that the restriction of $X^\sharp$ gives a bijection~\eqref{eq:Xshbij}, whose inverse
is the restriction of $Y^\sharp$ to $\{q\in\Proj(\Mcal)\mid q\le\ranproj(X)\}$.

To see that the bijection~\eqref{eq:Xshbij} is trace preserving, note that for a projection $p$
with $p\le\domproj(X)$, we have $\domproj(Xp)=p$, and by~\eqref{eq:taudomran},
\begin{equation}
\tau(p)=\tau(\ranproj(Xp))=\tau(X^\sharp(p)).
\end{equation}
An order preserving bijection between lattices is necessarily a lattice isomorphism.

Now we will show~\eqref{eq:Xshp}.
Using the form of the von Neumann
algebra generated by two projections as described in~\S\ref{subsec:twoproj},
we find that for any $p,q\in\Proj(\Mcal)$, we have
\begin{equation}
\ranproj(qp)=q-(1-p)\wedge q.
\end{equation}
Therefore,
\begin{align}
X^\sharp(p)&=\ranproj(Xp)=\ranproj(X\domproj(X)p) \\
&=\ranproj(X(\domproj(X)-(1-p)\wedge\domproj(X))) \\
&=X^\sharp(\domproj(X)-(1-p)\wedge\domproj(X))
\end{align}
and this implies
\begin{equation}\label{eq:tXp}
\tau(X^\sharp(p))=\tau(\domproj(X)-(1-p)\wedge\domproj(X)).
\end{equation}
Finally,~\eqref{eq:tauXshp} follows from~\eqref{eq:tXp} and~\eqref{eq:p1qp}.
\end{proof}

The next result concerns what may be termed {\em almost invariant subspaces} of operators.
We say $\Mcal$ is {\em diffuse} if it has no minimal nonzero projections.

\begin{propsubs}\label{prop:ais}
Assume that $\Mcal$ is diffuse.
Let $X$ be an operator affiliated with $\Mcal$ and let $0\le t\le\tau(\domproj(X))$ and $\eps>0$.
Then there are $p,q\in\Proj(\Mcal)$ such that $p,q\le\domproj(X)$, $\tau(p)=t$, $\tau(q)\le\eps$ and
\begin{equation}\label{eq:pvq}
X^\sharp(p)\le p\vee q.
\end{equation}
\end{propsubs}
\begin{proof}
Let $n$ be the least positive integer such that $t\le n\eps$.
We will proceed by induction on $n$.
If $n=1$, then take any $p\in\Proj(\Mcal)$ with $p\le\domproj(X)$ and $\tau(p)=t$ and let $q=X^\sharp(p)$.
Then $\tau(q)=\tau(p)=t\le\eps$.

For the induction step, suppose $n\ge2$ and $(n-1)\eps<t\le n\eps$.
By the induction hypothesis, there are $\pt,\qt\in\Proj(\Mcal)$ with $\pt\le\domproj(X)$, $\tau(\pt)=t-\eps$,
$\tau(\qt)<\eps$ and $X^\sharp(\pt)\le\pt\vee\qt$.
Replacing $\qt$ by $\qt-(\pt\wedge\qt)$, if necessary, we may without loss of generality assume
$\qt\wedge\pt=0$.
Adding something from $\domproj(X)-(\pt\vee\qt)$ to $\qt$, if necessary, we may also without loss of generality
assume $\tau(\qt)=\eps$.
Now let $p=\pt\vee\qt$.
Then $\tau(p)=t$ and
\begin{equation}
X^\sharp(p)=X^\sharp(\pt)\vee X^\sharp(\qt)
\le\pt\vee\qt\vee X^\sharp(\qt)=p\vee X^\sharp(\qt).
\end{equation}
Let $q=X^\sharp(\qt)$.
Then $\tau(q)=\tau(\qt)=\eps$ and~\eqref{eq:pvq} holds.
\end{proof}

\subsection{The complementary idempotents of two projections}
\label{subsec:idems}

In this subsection, we consider the idempotent affiliated operators associated to two
projections $e_1$ and $e_2$ in a finite von Neumann algebra $\Mcal$.
We fix a normal faithful tracial state $\tau$ on $\Mcal$ and, for convenience, we regard $\Mcal$
as acting on $\HEu:=L^2(\Mcal,\tau)$ in the GNS--representation.

We define possibly unbounded operators $E(e_1,e_2)$ and $E(e_2,e_1)$, both with domain
\begin{multline}
(1-(e_1\vee e_2))\HEu+e_1\HEu+e_2\HEu \\
=(1-(e_1\vee e_2))\HEu+(e_1-e_1\wedge e_2)\HEu+(e_2-e_1\wedge e_2)\HEu+(e_1\wedge e_2)\HEu,
\end{multline}
as follows.
For ease of notation, we write $E_1$ for $E(e_1,e_2)$ and $E_2$ for $E(e_2,e_1)$.
We set
\begin{equation}
E_i(\eta+\xi_1+\xi_2+\zeta)=\xi_i+\zeta
\end{equation}
if $\eta\in(1-(e_1\vee e_2))\HEu$, $\xi_j\in(e_j-e_1\wedge e_2)\HEu$, ($j=1,2$)
and $\zeta\in(e_1\wedge e_2)\HEu$.
It is clear that $E_i$ is well defined.

\begin{lemmasubs}\label{lem:Ei}
\renewcommand{\labelenumi}{(\roman{enumi})}
\begin{enumerate}
\item
Each operator $E_i$ is closed, affiliated with the von Neumann algebra $W^*(\{1,e_1,e_2\})$
generated by $e_1$ and $e_2$, and idempotent.
\item
We have
\begin{align}
\ranproj(E_i)&=e_i, \label{eq:ranEi} \\
\kerproj(E_i)&=1-e_1\vee e_2+e_{i'}-e_1\wedge e_2, \label{eq:kerEi} \\
\domproj(E_i)&=e_1\vee e_2-e_{i'}+e_1\wedge e_2, \label{eq:domEi}
\end{align}
where $\{i,i'\}=\{1,2\}$, and
\begin{equation}\label{eq:E1pE2}
E_1+E_2=(e_1\vee e_2)+(e_1\wedge e_2).
\end{equation}
\item
Let $f\in\Mcal$ be a projection with $f\le e_1\vee e_2$.
Then 
\begin{equation}\label{eq:fsharp}
f\le E_1^\sharp(f)\vee E_2^\sharp(f)\vee(e_1\wedge e_2-(1-f)\wedge e_1\wedge e_2).
\end{equation}
\end{enumerate}
\end{lemmasubs}
\begin{proof}
To show that $E_i$ is closed, 
(taking $i=1$),
if $h^{(n)}\in\dom(E_1)$ converges to $h\in\HEu$ 
and if $E_1(h^{(n)})$ converges to $y\in\HEu$, then we may write
\begin{equation}
h^{(n)}=\eta^{(n)}+\xi_1^{(n)}+\xi_2^{(n)}+\zeta^{(n)},
\end{equation}
where $\eta^{(n)}=(1-(e_1\vee e_2))h^{(n)}$, $\zeta^{(n)}=(e_1\wedge e_2)h^{(n)}$ and where
$\xi_j^{(n)}\in(e_j-e_1\wedge e_2)\HEu$.
We then have convergence:
\begin{align}
\eta^{(n)}&\rightarrow(1-(e_1\vee e_2))h \\
\zeta^{(n)}&\rightarrow(e_1\wedge e_2)h \\
\xi_1^{(n)}=E_1(h^{(n)})-\zeta^{(n)}&\rightarrow y-(e_1\wedge e_2)h\in(e_1-e_1\wedge e_2)\HEu.
\end{align}
Thus, we also get convergence
\begin{equation}
\xi_2^{(n)}\rightarrow z:=(e_1\vee e_2-e_1\wedge e_2)(h)-y\in(e_2-e_1\wedge e_2)\HEu.
\end{equation}
Consequently, we have $h=(1-(e_1\vee e_2))h+y+z+(e_1\wedge e_2)h$ and we conclude
$E_1(h)=y$.
So $E_1$ is closed.

By the analysis in section~\ref{subsec:twoproj},
we have
\begin{equation}
W^*(\{1,e_1,e_2\})=
\smdp{\Cpx}{\gamma_{11}}{e_1\wedge e_2}
  \oplus\smdp{\Cpx}{\gamma_{10}}{e_1\wedge(1-e_2)}
  \oplus L^\infty(\mu)\otimes M_2(\Cpx)
  \oplus\smdp{\Cpx}{\gamma_{01}}{(1-e_1)\wedge e_2}
  \oplus\smdp{\Cpx}{\gamma_{00}}{(1-e_1)\wedge(1-e_2)},
\end{equation}
for some measure $\mu$ on $(0,1)$, with
\begin{align}
e_1&=1\oplus 1\oplus\begin{pmatrix}1&0\\0&0\end{pmatrix}\oplus0\oplus0 \\
e_2&=1\oplus 0\oplus\begin{pmatrix}t&\sqrt{t(1-t)}\\\sqrt{t(1-t)}&1-t\end{pmatrix}\oplus1\oplus0.
\end{align}
Now compressing by the appropriate central projections, we easily see that $E_1$ and $E_2$ are
limits in s.o.t.\ of elements of $W^*(\{1,e_1,e_2\})$, hence are affiliated with
this von Neumann algebra and, in fact, can be written as
\begin{align}
E_1&=1\oplus 1\oplus\begin{pmatrix}1&-\sqrt{t/(1-t)}\\0&0\end{pmatrix}\oplus0\oplus0 \label{eq:E1mat} \\
E_2&=1\oplus 0\oplus\begin{pmatrix}0&\sqrt{t/(1-t)}\\0&1\end{pmatrix}\oplus1\oplus0, \label{eq:E2mat}
\end{align}
where this has the obvious meaning.
It is clear from their definition that $E_1$ and $E_2$ are idempotent.
This shows~(i).

For~(ii),
we see from the definition that
\begin{equation}
\ker(E_i)=(1-e_1\vee e_2)\HEu+(e_{i'}-e_1\wedge e_2)\HEu,
\end{equation}
so we get~\eqref{eq:kerEi} and~\eqref{eq:domEi}.
Also,~\eqref{eq:ranEi} is obvious, while~\eqref{eq:E1pE2}
follows from~\eqref{eq:E1mat} and~\eqref{eq:E2mat}.

For~(iii), it is straightforward to see that
$f\HEu\cap(e_1\HEu+e_2\HEu)$ is dense in $f\HEu$, so letting $r$ be the projection
on the right--hand--side of~\eqref{eq:fsharp}, it will suffice to show
$f\HEu\cap(e_1\HEu+e_2\HEu)\subseteq r\HEu$.
Let $h\in f\HEu\cap(e_1\HEu+e_2\HEu)$.
Then $h=\xi_1+\xi_2+(e_1\wedge e_2)h$, for $\xi_i\in(e_i-e_1\wedge e_2)\HEu$.
We have
\begin{equation}
\xi_i+(e_1\wedge e_2)h=E_i(h)\in E_i^\sharp(f)\HEu,
\end{equation}
while using Lemma~\ref{lem:pqH}, we have
\begin{equation}
(e_1\wedge e_2)h\in(e_1\wedge e_2)f\HEu\subseteq(e_1\wedge e_2-(1-f)\wedge e_1\wedge e_2)\HEu.
\end{equation}
So $h\in r\HEu$.
\end{proof}

\section{Irreducible Horn triples}
\label{sec:Horn}

Horn's inequalities in the $n\times n$ matrices are of the form
\begin{equation}\label{eq:IJK}
\sum_{i\in I}\alpha_i+\sum_{j\in J}\beta_j\ge\sum_{k\in K}\gamma_k.
\end{equation}
for certain triples $(I,J,K)$ of subsets of $\{1,\ldots,n\}$.
In~\cite{H62}, Horn defined sets $T^n_r$ of triples $(I , J, K )$ of subsets of 
$\{1,\ldots,n\}$ of the same 
cardinality $r$, by the following recursive procedure.
By convention, a subset $I$ of $\{1,\ldots,n\}$ is indexed in increasing order:
\begin{equation}\label{eq:Iindexed}
I=\{i_1,\ldots,i_r\},\qquad i_1<i_2<\cdots<i_r.
\end{equation}
Set 
\begin{equation}\label{eq:Unr}
U^n_r = \bigg\{(I , J, K )\bigg|\sum_{i\in I}i+\sum_{j\in J}j=\sum_{k\in K}k+\frac{r(r+1)}2\bigg\}.
\end{equation}
When $r = 1$, set  $T^n_1 = U^n_1$.
Otherwise, let
\begin{equation}\label{eq:Tnr}
\begin{aligned}
T^n_r = \bigg\{(I,J,K)\in U^n_r\bigg|
    \sum_{f\in F}i_f+\sum_{g\in G}j_g\leq\sum_{h\in H}k_h+\frac{p(p+1)}2,& \\
\text{ for all }p < r\text{ and }(F,G,H)\in T^r_p\;\;\;&\bigg\}.
\end{aligned}
\end{equation}

We will consider a reformulation of Horn's sets $T_r^n$, which was used also in~\cite{TT74}.
Let $\sigma_n$ be the permutation of $\{1,\ldots,n\}$ given by $\sigma_n(i)=n+1-i$.
Thus, if $I$ is indexed as in~\eqref{eq:Iindexed} and if we use the same convention for indexing
$\sigma_n(I)$, namely
\begin{equation}
\sigma_n(I)=\{\itil_1,\ldots,\itil_r\},\qquad \itil_1<\itil_2<\cdots<\itil_r,
\end{equation}
then $i_j=n+1-\itil_{r+1-j}$.
We let
\begin{equation}
\Tt^n_r=\{(\sigma_n(I),\sigma_n(J),K)\mid (I,J,K)\in T_r^n\}.
\end{equation}
Reformulating Horn's definition, these sets are recursively defined as follows.
Let $\Ut_r^n$ be the set consising of triples $(I,J,K)$ of subsets of $\{1,\ldots,n\}$
with $|I|=|J|=|K|=r$ by
\begin{equation}\label{eq:Ut}
\Ut_r^n=\bigg\{(I,J,K)\bigg|\sum_{i\in I}i+\sum_{j\in J}j+\sum_{k\in K}k=\frac{r(4n-r+3)}2\bigg\}.
\end{equation}
If $r=1$, then we have $\Tt_r^n=\Ut_r^n$, while for $r\in\{2,\ldots,n-1\}$, we have
\begin{equation}\label{eq:Tt}
\begin{aligned}
\Tt^n_r = \bigg\{(I,J,K)\in U^n_r\bigg|
    \sum_{f\in F}i_f+\sum_{g\in G}j_g+\sum_{h\in H}k_h\ge\frac{p(4n-p+3)}2,& \\
\text{ for all }p < r\text{ and }(F,G,H)\in\Tt^r_p\;\;\;&\bigg\}.
\end{aligned}
\end{equation}
Now, for $(I,J,K)\in\Tt_r^n$, the corresponding Horn inequality is
\begin{equation}\label{eq:IJKt}
\sum_{i\in I}\alpha_{n+1-i}+\sum_{j\in J}\beta_{n+1-j}\ge\sum_{k\in K}\gamma_k.
\end{equation}

This reformulation of the Horn inequalities
has certain advantages.
As is apparent from the symmetry of~\eqref{eq:Ut} and~\eqref{eq:Tt},
the set $\Tt_r^n$ is invariant under permuting
the three sets $I$, $J$ and $K$.
Moreover, Proposition~\ref{prop:flagHorn} and the reduction procedure
resulting from Lemma~\ref{lem:TT1}
are more natural in this alternative expression of the Horn
inequalities.

In~\cite{TT74}, S.~Therianos and R.C.~Thompson proved that many
Horn inequalities in $T^n_r$ can be reduced to inequalities in $T^{n-1}_r$.
We will prove that similar results hold in finite von Neumann algebras.

For future use in this section, we 
record the following integration--by--parts formula for Riemann--Stieltjes
integrals, which is well known and easily proved.
\begin{lemma}\label{lem:RS}
Let $f:[0,1]\to\Reals$ be continuous and let $\lambda:[0,1]\to\Reals$ be monotone
and assume $\lambda$ is (one--sided) continuous at $0$ and $1$.
Then the Riemann--Stieltjes integrals $\int_0^1\lambda(t)df(t)$ and $\int_0^1f(t)d\lambda(t)$ exist,
and we have
\begin{equation}
\int_0^1\lambda(t)df(t)=\lambda(1)f(1)-\lambda(0)f(0)-\int_0^1f(t)d\lambda(t).
\end{equation}
\end{lemma}

\begin{defi}
Let $\Mcal$ be a diffuse, finite von Neumann algebra with a fixed faithful normal tracial state $\tau$.
A {\em flag} in $\Mcal$ is a linearly ordered family $e=(e_t)_{0\le t\le1}$ of projections
in $\Mcal$ such that $\tau(e_t)=t$ for all $t$.

A {\em superflag} in $\Mcal$ is a family $f=(f_t)_{0\le t\le1}$ of projections in $\Mcal$ 
such that $f_s\le f_t$ whenever $s\le t$ and $\tau(f_t)\ge t$ for all $t\in[0,1]$.
\end{defi}

\begin{prop}\label{prop:subflag}
If $f=(f_t)_{0\le t\le1}$ is a superflag in $\Mcal$, then there is a flag $e=(e_t)_{0\le t\le 1}$
in $\Mcal$ such that $e_t\le f_t$ for all $t$.
\end{prop}
\begin{proof}
Let $\mathcal{S}$ be the set of sets of projections of $\mathcal{M}$ such that
for any $S\in\mathcal{S}$ and any $t\in [0,1]$, $f_t\in S$, and for all $p,q\in S$, either $p\geq q$
or $q\geq p$.

The set $\mathcal{S}$ is a Zorn inductive set for the obvious order given by inclusion.
Let $\St$ be a maximal element. 
The set of values $\tau (p),p\in\St$ is closed by maximality.
Suppose, to obtain a contradiction, that this set is not all of $[0,1]$.
Let $t\in[0,1]$ be a value that is not attained,
and let 
\begin{equation}
\begin{aligned}
t_-&=\sup\{\tau(p)\mid p\in\St,\,\tau(p)<t\} \\
t_+&=\inf\{\tau(p)\mid p\in\St,\,\tau(p)>t\},
\end{aligned}
\end{equation}
so that we have
$t_-<t<t_+$.
Let $p_{\pm}$ in $\St$ such that
$\tau (p_{\pm})=t_{\pm}$.
By elementary properties of finite diffuse von Neuman algebras, there is a projection $p\in\Mcal$
between $p_{-}$ and $p_{+}$ such that $\tau (p)=t$.
This contradicts maximality of $\St$.

To construct the flag, for each $t$, let $e_t$ be the unique $p\in\St$ such that $\tau(p)=t$.
\end{proof}

Property P$_n$ below is the von Neumann algebra analogue of Therianos and Thompson's property
of the same name (which applied to matrices).

\begin{defi}\label{def:Pn}
Let $r$ and $n$ be positive integers with $r\le n$.
Consider a triple $(I,J,K)$ of subsets of $\{1,\ldots,n\}$, each having cardinality $r$.
Write
\begin{align}
I=\{i_1,\ldots,i_r\},&\qquad i_1< i_2<\cdots< i_r \label{eq:Ielts} \\
J=\{j_1,\ldots,j_r\},&\qquad j_1< j_2<\cdots< j_r \\
K=\{k_1,\ldots,k_r\},&\qquad k_1< k_2<\cdots< k_r. \label{eq:Kelts}
\end{align}
We say $(I,J,K)$ has {\em property P$_n$} if whenever $e$, $f$ and $g$ are flags in any finite von Neumann algebra
$(\Mcal,\tau)$, 
there exists a projection $p\in\Mcal$ such that
\begin{equation}\label{eq:taup=}
\tau(p)=\frac rn
\end{equation}
and for all $\ell\in\{1,2,\ldots, r\}$, we have
\begin{align}
\tau(e_{\frac{i_\ell}n}\wedge p)&\ge\frac \ell n \label{eq:ewp} \\[0.7ex]
\tau(f_{\frac{j_\ell}n}\wedge p)&\ge\frac \ell n \\[0.7ex]
\tau(g_{\frac{k_\ell}n}\wedge p)&\ge\frac \ell n\,. \label{eq:gwp}
\end{align}
We say that $(I,J,K)$ as {\em property AP$_n$} if whenever $e$, $f$ and $g$ are flags in
any finite von Neumann algebra $(\Mcal,\tau)$, 
and whenever $\eps>0$, there is a projection $p\in\Mcal$ such that 
\begin{equation}
\tau(p)\le\frac rn+\eps
\end{equation}
and for all $\ell\in\{1,2,\ldots, r\}$, the inequalities \eqref{eq:ewp}--\eqref{eq:gwp} hold.
\end{defi}

The following result is analogous to well--known facts in $n\times n$ matrices.
The proof in the case of property P$_n$ was can easily be found in~\cite{BL01}, and the approximate result
follows straightforwardly.
For convenience, we write a proof pointing to the appropriate parts of~\cite{BL01}.

\begin{prop}\label{prop:flagHorn}
If $(I,J,K)\in\Tt_r^n$ has property P$_n$ or, more generally, property AP$_n$,
then the Horn inequality corresponding to $(\sigma_n(I),\sigma_n(J),K)$ holds
in every finite von Neumann algebra.
\end{prop}
\begin{proof}
First suppose that $(I,J,K)$ has property P$_n$.
Let $\It=\sigma_n(I)$, $\Jt=\sigma_n(J)$.
We must show, given any finite von Neumann algebra $(\Mcal,\tau)$ and any $a,b\in\Mcal_{s.a.}$, that we have
\begin{equation}\label{eq:Hab}
\int_{\omega_\It}\lambda_a(t)\,dt+\int_{\omega_\Jt}\lambda_b(t)\,dt\ge\int_{\omega_K}\lambda_{a+b}(t)\,dt.
\end{equation}
But $\omega_{\It}=1-\omega_I:=\{1-t\mid t\in\omega_I\}$ and $\lambda_a(t)=-\lambda_{-a}(1-t)$,
so letting $x=-a$, $y=-b$ and $z=a+b$, the inequality~\eqref{eq:Hab} becomes
\begin{equation}
\int_{\omega_I}\lambda_x(t)\,dt+\int_{\omega_J}\lambda_y(t)\,dt+\int_{\omega_K}\lambda_z(t)\,dt\le0,
\end{equation}
which must be proved for all $x,y,z\in\Mcal_{s.a}$ such that $x+y+z=0$.

Let $E_x$, $E_y$ and $E_z$ be the spectral measures of $x$, $y$ and $z$.
As described on page~115 of~\cite{BL01}, there are flags $e$, $f$ and $g$ in $\Mcal$ such that
\begin{align}
x&=\int_0^1\lambda_x(t)\,de_t \\
y&=\int_0^1\lambda_y(t)\,df_t \\
z&=\int_0^1\lambda_z(t)\,dg_t
\end{align}
where these integrals are the operator--valued analogues of Riemann--Stieltjes integrals,
and, for all $t\in[0,1]$, we have
\begin{align}
E_x((\lambda_x(t),\infty))&\le e_t \label{eq:Eex} \\
E_y((\lambda_y(t),\infty))&\le f_t \label{eq:Efy} \\
E_z((\lambda_z(t),\infty))&\le g_t\,. \label{eq:Egz}
\end{align}
Consider the nondecreasing function $W_I$ on $[0,1]$ which at $t$ takes value equal
to the Lebesgue measure of $\omega_I\cap[0,t]$.
Then $W_I$ is piecewise linear, has slope $1$ on intervals $(\frac{i-1}n,\frac in)$ for $i\in I$
(thus, at points of $\omega_I$)
and has slope $0$ elsewhere.
Furthermore,
\begin{equation}\label{eq:omegaW}
\int_{\omega_I}\lambda_x(t)\,dt=\int_0^1\lambda_x(t)\,dW_I(t),
\end{equation}
where the right--hand--side is the Riemann--Stieltjes integral.

Using that $(I,J,K)$ has property P$_n$, let $p\in\Mcal$ be a projection
satisfying~\eqref{eq:taup=} and~\eqref{eq:ewp}--\eqref{eq:gwp}.
Using~\eqref{eq:ewp}, we get
\begin{equation}\label{eq:taupeW}
\tau(p\wedge e_t)\ge W_I(t)
\end{equation}
whenever $t=\frac in$ with $i\in I$.
Moreover, taking $0\le s\le t\le 1$ and using Proposition~\ref{prop:taup}, we have
\begin{align}
\tau(p\wedge e_s)&=\tau((p\wedge e_t)\wedge e_s) \\
&=\tau(p\wedge e_t)+\tau(e_s)-\tau((p\wedge e_t)\vee e_s) \\
&\ge\tau(p\wedge e_t)+\tau(e_s)-\tau(e_t) \\
&=\tau(p\wedge e_t)-(t-s).  
\end{align}
This implies both that $\tau(p\wedge e_t)$ is a continous function of $t$ and that~\eqref{eq:taupeW} holds
at all points $t\in\omega_I$ and, of course, at $t=0$, where both sides are zero.
However, since $W_I(t)$ is constant elsewhere and since $\tau(p\wedge e_t)$ is increasing,
the inequality~\eqref{eq:taupeW} holds for all $t\in[0,1]$.
We define $\lambda_x(1)$ to make $\lambda_x$ continuous from the right at $1$.
Using Lemma~\ref{lem:RS} and that we have
\begin{gather}
W_I(0)=0=\tau(p\wedge e_0) \\
W_I(1)=\frac rn=\tau(p\wedge e_1)
\end{gather}
we get
\begin{align}
\int_0^1\lambda_x(t)\,dW_I(t)&=\lambda_x(1)\frac rn+\int_0^1W_I(t)\,d(-\lambda_x)(t) \label{eq:Wtauint1} \\
&\le\lambda_x(1)\frac rn+\int_0^1\tau(p\wedge e_t)\,d(-\lambda_x)(t) \\
&=\int_0^1\lambda_x(t)\,d(\tau(p\wedge e_t)), \label{eq:Wtauint3} 
\end{align}
where the above inequality is because $-\lambda_x$ is nondecreasing and the inequality~\eqref{eq:taupeW} holds.
However, by Proposition~2.1 of~\cite{BL01}, we have
\begin{equation}
\int_0^1\lambda_x(t)\,d\tau(p\wedge e_t)\le\tau(xp).
\end{equation}
Putting this together with~\eqref{eq:omegaW} and~\eqref{eq:Wtauint1}--\eqref{eq:Wtauint3}, we have
\begin{equation}
\int_{\omega_I}\lambda_x(t)\,dt\le\tau(xp).
\end{equation}
Arguing similarly for $y$ and $z$, we get
\begin{equation}
\int_{\omega_I}\lambda_x(t)\,dt+\int_{\omega_J}\lambda_y(t)\,dt+\int_{\omega_K}\lambda_z(t)\,dt
\le\tau((x+y+z)p)=0,
\end{equation}
as required.

Now suppose $(I,J,K)$ has property AP$_n$.
Letting $\eps>0$, we may argue as above, except that instead of being able to choose $p$ so that
\eqref{eq:taup=} and~\eqref{eq:ewp}--\eqref{eq:gwp} are satisfied, in place of the equality~\eqref{eq:taup=}
we may only assume
\begin{equation}
\tau(p)\le\frac rn+\eps.
\end{equation}
Now instead of getting
$\int_0^1\lambda_x(t)\,dW_I(t)\le\int_0^1\lambda_x(t)\,d(\tau(p\wedge e_t))$
as we did in~\eqref{eq:Wtauint1}--\eqref{eq:Wtauint3}, we get
\begin{equation}
\int_0^1\lambda_x(t)\,dW_I(t)\le|\lambda_x(1)|\eps+\int_0^1\lambda_x(t)\,d(\tau(p\wedge e_t)).
\end{equation}
Using $|\lambda_x(1)|\le\|x\|$ and arguing as above, we get
\begin{equation}
\int_{\omega_I}\lambda_x(t)\,dt+\int_{\omega_J}\lambda_y(t)\,dt+\int_{\omega_K}\lambda_z(t)\,dt
\le\eps(\|x\|+\|y\|+\|z\|).
\end{equation}
Letting $\eps$ tend to zero yields the desired inequality.
\end{proof}

The following lemma is an analogue for finite von Neumann algebras
of Lemma~1 of~\cite{TT74}.
We will use it to reduce the set of Horn inequalities that must be verified in finite von Neumann algebras.

Let
\begin{equation}
h_x(y)=
\begin{cases}
0,&y\le x \\
1,&y>x.
\end{cases}
\end{equation}

\begin{lemma}\label{lem:TT1}
Let $1\le r\le n$ be integers.
Let $(I,J,K)$ be a triple of subsets of $\{1,\ldots,n\}$ satisfying~\eqref{eq:Ielts}--\eqref{eq:Kelts}
and assume this triple has property P$_n$, respectively, property AP$_n$.
Also set $i_0=j_0=k_0=0$.
Suppose $u,v,w\in\{0,1,\ldots,r\}$ are such that
\begin{equation}\label{eq:iujvkw}
i_u+j_v+k_w\le n.
\end{equation}
Set
\begin{equation}\label{eq:ijky}
\begin{aligned}
i'_y&=i_y+h_u(y) \\
j'_y&=j_y+h_v(y) \\
k'_y&=k_y+h_w(y)
\end{aligned}
\qquad(y\in\{1,\ldots,r\}).
\end{equation}
and let
\begin{equation}
I'=\{i'_1,\ldots,i'_r\},\quad
J'=\{j'_1,\ldots,j'_r\},\quad
K'=\{k'_1,\ldots,k'_r\}.
\end{equation}
Then $(I',J',K')$ has property P$_{n+1}$, respectively, property AP$_{n+1}$.
\end{lemma}
\begin{proof}
Let $(\Mcal,\tau)$ be a diffuse, finite von Neumann algebra and let $e$, $f$ and $g$ be any flags in $\Mcal$.
Suppose $(I,J,K)$ has property P$_n$.
{}From~\eqref{eq:iujvkw}, we have
\begin{equation}
\tau(e_{\frac{i_u}{n+1}}\vee f_{\frac{j_v}{n+1}}\vee g_{\frac{k_w}{n+1}})\le\frac n{n+1}\,.
\end{equation}
Let $q\in\Mcal$ be a projection such that $\tau(q)=\frac n{n+1}$ and
\begin{equation}
e_{\frac{i_u}{n+1}}\vee f_{\frac{j_v}{n+1}}\vee g_{\frac{k_w}{n+1}}\le q.
\end{equation}
Then $q\wedge e_t=e_t$ if $t\le\frac{i_u}{n+1}$ and, for all $t$,
$\tau(q\wedge e_t)\ge t-\frac1{n+1}$.
Similar results hold for $f$ and $g$.
Define
\begin{align}
e'_t&=
\begin{cases}
e_{\frac{nt}{n+1}}\,,&0\le t\le\frac{i_u}n \\
e_{\frac{nt+1}{n+1}}\wedge q,&\frac{i_u}n<t\le1
\end{cases} \\
f'_t&=
\begin{cases}
f_{\frac{nt}{n+1}}\,,&0\le t\le\frac{j_v}n \\
f_{\frac{nt+1}{n+1}}\wedge q,&\frac{j_v}n<t\le1,
\end{cases} \\
g'_t&=
\begin{cases}
g_{\frac{nt}{n+1}}\,,&0\le t\le\frac{k_w}n \\
g_{\frac{nt+1}{n+1}}\wedge q,&\frac{k_w}n<t\le1.
\end{cases}
\end{align}
Then in the cut--down von Neumann algebra $q\Mcal q$, equipped with the rescaled trace
$\frac{n+1}n\tau|_{q\Mcal q}$, $e'$, $f'$ and $g'$ are superflags.
Invoking Proposition~\ref{prop:subflag}, let $\et$, $\ft$ and $\gt$ be flags in
$q\Mcal q$ such that $\et_t\le e'_t$, $\ft_t\le f'_t$ and $\gt_t\le g'_t$ for all $t\in[0,1]$.
Then we have
\begin{gather}
\et_t=e'_t=e_{\frac{nt}{n+1}}\,,\qquad(0\le t\le\frac{i_u}n) \label{eq:etteq} \\[0.7ex]
\et_t\le e'_t=e_{\frac{nt+1}{n+1}}\wedge q\,,\qquad(\frac{i_u}n<t\le 1). \label{eq:ettle}
\end{gather}
By the assumption that $(I,J,K)$ has property P$_n$, there is a projection $p\in q\Mcal q$ such that
\begin{equation}\label{eq:rstauple}
\frac{n+1}n\tau(p)\le\frac rn
\end{equation}
and, for all $y\in\{1,\ldots,r\}$, we have
\begin{align}
\frac{n+1}n\tau(\et_{\frac{i_y}n}\wedge p)&\ge\frac yn \label{eq:tauetge} \\[0.7ex]
\frac{n+1}n\tau(\ft_{\frac{j_y}n}\wedge p)&\ge\frac yn \\[0.7ex]
\frac{n+1}n\tau(\gt_{\frac{k_y}n}\wedge p)&\ge\frac yn\,.
\end{align}
We will show that $p$ is the desired projection for $(I',J',K')$ to have property P$_{n+1}$.
We have
\begin{equation}\label{eq:tauple}
\tau(p)\le\frac r{n+1}.
\end{equation}
If $y\in\{1,\ldots,u\}$, then $i'_y=i_y$ and using~\eqref{eq:etteq} with $t=\frac{i_y}n$
and~\eqref{eq:tauetge}, we get
\begin{equation}\label{eq:tauewp}
\tau(e_{\frac{i'_y}{n+1}}\wedge p)\ge\frac y{n+1},
\end{equation}
while if $y\in\{u+1,\ldots,r\}$, then $i'_y=i_y+1$, so using~\eqref{eq:ettle}
with $t=\frac{i_y}n$ and that $p\le q$,
we have
\begin{equation}
\et_{\frac{i_y}n}\wedge p\le e_{\frac{i'_y}{n+1}}\wedge p,
\end{equation}
and from~\eqref{eq:tauetge} we get~\eqref{eq:tauewp} also in this case.
In a similar manner, we get
\begin{equation}
\begin{aligned}
\tau(f_{\frac{j'_y}{n+1}}\wedge p)&\ge\frac y{n+1} \\
\tau(g_{\frac{k'_y}{n+1}}\wedge p)&\ge\frac y{n+1}
\end{aligned}
\end{equation}
for all $y\in\{1,\ldots,r\}$.
Thus, $(I',J',K')$ has property P$_{n+1}$.

In the case that $(I,J,K)$ has only property AP$_n$, the same argument applies, except that, given $\eps>0$,
instead of~\eqref{eq:rstauple} we get
\begin{equation}
\frac{n+1}n\tau(p)\le\frac rn+\eps
\end{equation}
and this yields
\begin{equation}
\tau(p)\le\frac r{n+1}+\frac n{n+1}\eps.
\end{equation}
\end{proof}

\begin{remark}\label{rem:red}
Lemma~\ref{lem:TT1} provides a reduction procedure with respect to properties P$_n$ and AP$_n$,
in the following sense.
Let $(I,J,K)\in\Tt^n_r$.
Suppose there are $u,v,w\in\{0,\ldots,r\}$ such that all of the following four statements hold:
\begin{align}
u=r\text{ or }&i_{u+1}-i_u\ge2 \label{eq:idiff} \\
v=r\text{ or }&j_{v+1}-j_v\ge2 \label{eq:jdiff} \\
w=r\text{ or }&k_{w+1}-k_w\ge2 \label{eq:kdiff} \\
i_u+j_v+&k_w\le n-1, \label{eq:ijksumsmall}
\end{align}
where again we set $i_0=j_0=k_0=0$.
Then Lemma~\ref{lem:TT1} applies, and
to verify that $(I,J,K)$ has property P$_n$, respectively, AP$_n$,
it will suffice to show that $(\It,\Jt,\Kt)$ has property P$_{n-1}$, respectively, AP$_{n-1}$, where
\begin{equation}\label{eq:ItJtKt}
\It=(\itil_1,\ldots,\itil_r),\qquad\Jt=(\jt_1,\ldots,\jt_r),\qquad\Kt=(\kt_1,\ldots,\kt_r)
\end{equation}
are given by
\begin{align}
\itil_p&=
\begin{cases}
i_p\,,&1\le p\le u \\
i_p-1,& u<p\le r,
\end{cases} \label{eq:it} \\
\jt_p&=
\begin{cases}
j_p\,,&1\le p\le v \\
j_p-1,& v<p\le r,
\end{cases} \\
\kt_p&=
\begin{cases}
k_p\,,&1\le p\le w \\
k_p-1,& w<p\le r.
\end{cases} \label{eq:kt}
\end{align}
In fact, we will only concern ourselves with this reduction procedure under the additional
hypothesis
\begin{equation}\label{eq:uvwr}
u+v+w=r,
\end{equation}
which is quite natural because it insures that $(I,J,K)\in\Ut^n_r$ implies $(\It,\Jt,\Kt)\in\Ut^{n-1}_r$.
In fact, we will soon
show that $(I,J,K)\in\Tt^n_r$ implies $(\It,\Jt,\Kt)\in\Tt^{n-1}_r$
for this reduction procedure under the additional hypothesis~\eqref{eq:uvwr},
and, even more, that Littlewood--Richardson coefficients are preserved.
\end{remark}

An important part of the solution of Horn's conjecture was to relate Horn's triples $(I,J,K)\in T^n_r$
to Littlewood--Richardson coefficients.
If $I$ is a set of $r$ distinct positive integers, written as in~\eqref{eq:Ielts}, then we let
\begin{equation}\label{eq:rho}
\rho_r(I)=(i_r-r,i_{r-1}-(r-1),\ldots,i_1-1).
\end{equation}
Note that $\rho_r(I)=(\lambda_1,\lambda_2,\ldots,\lambda_r)$ consists of integers satisfying
\begin{equation}\label{eq:lamdec}
\lambda_1\ge\lambda_2\ge\cdots\ge\lambda_r\ge0.
\end{equation}
We let $\Nats_{0,\ge}^r$ denote the set of $r$--tuples $(\lambda_1,\ldots,\lambda_r)$
of integers satisfying~\eqref{eq:lamdec}, and note that $\rho_r$ is a bijection from the set of subsets
of $\Nats$ having cardinality $r$ onto $\Nats_{0,\ge}^r$.
For $n,r\in\Nats$, $n\ge r$, let
\begin{equation}
\Lambda^n_r=\{(\lambda,\mu,\nu)=(\rho_r(I),\rho_r(J),\rho_r(K))\in(\Nats_{0,\ge}^r)^3\mid
(I,J,K)\in T^n_r\},
\end{equation}
where $T^n_r$ is the usual set of Horn triples.
Using Thm.\ 12 of~\cite{F00}, we easily see
\begin{equation}\label{eq:Lambdanr}
\Lambda^n_r=\bigg\{(\lambda,\mu,\nu)\in(\Nats_{0,\ge}^r)^3\bigg|
\;\sum_{p=1}^r(\lambda_p+\mu_p)=\sum_{p=1}^r\nu_p,\quad
\nu_1\le n-r,\quad c_{\lambda,\mu}^\nu\ne0\;\bigg\},
\end{equation}
where $c_{\lambda,\mu}^\nu$ is the Littlewood--Richardson coefficient, which is a nonnegative integer.
(See~\cite{F00} for more about these.)

The map
\begin{equation}\label{eq:bij}
\Phi^n_r:(I,J,K)\mapsto(\rho_r(\sigma_n(I)),\rho_r(\sigma_n(J)),\rho_r(K))
\end{equation}
is an injective map from the set of triples of subsets of $\{1,2,\ldots,n\}$, each with cardinality $r$,
to $(\Nats_{0,\ge}^r)^3$
and restricts to a bijection from $\Tt^n_r$ onto $\Lambda^n_r$.

\begin{defi}\label{def:LRc}
Let $(I,J,K)$ be a triple of subsets of $\{1,\ldots,n\}$, with $|I|=|J|=|K|=r$.
The {\em Littlewood--Richardson coefficient} of $(I,J,K)$, denoted $c^{(n)}(I,J,K)$,
is equal to the Littlewood--Richardson coefficient $c_{\lambda,\mu}^\nu$, where
$(\lambda,\mu,\nu)=\Phi^n_r((I,J,K))$.
\end{defi}

As already remarked, if $(I,J,K)\in\Tt^n_r$ then all the triples
\begin{equation}
(I,K,J),\quad(J,I,K),\quad(J,K,I),\quad(K,I,J),\quad(K,J,I)
\end{equation}
are also in $\Tt^n_r$.
So at least the property $c^{(n)}(I,J,K)>0$ is invariant under permuting the three sets $I$, $J$ and $K$.
We now show that the Littlewood--Richardson ceofficient is itself invariant.
\begin{lemma}\label{lem:cIJK}
The Littlewood--Richardson coefficient $c^{(n)}(I,J,K)$ is invariant under permutation of the three arguments.
\end{lemma}
\begin{proof}
By definition,
$c^{(n)}(I,J,K)=c_{\lambda,\mu}^{\nu}$ is the number of components of type 
$V_{\nu}$ that one finds in $V_{\lambda}\otimes V_{\mu}$,
where $V_{\lambda},V_{\mu},V_{\nu}$ are irreducible rational representations
of $GL(r,\mathbb{C})$.
In other words, it is 
\begin{equation}
\dim\Hom_{GL(r,\mathbb{C})}(V_{\nu},V_{\lambda}\otimes V_{\mu}).
\end{equation}

Observe that the contragredient representation  of $V_{\nu}$ is the
representation of highest weights $(1-k_1,\ldots ,r-k_r)$. 
Following the representation theory conventions,
we shall denote by
$\bar{V}_{\nu}$ this representation.

The fact that $V_{\nu}$ is irreducible implies by Schur's lemma that
$\bar{V}_{\nu}\otimes V_{\nu}$ contains one and only one copy of the trivial representation 
$\varepsilon$ (of highest weight $(0,0,\ldots ,0)$).

Observe also that the determinant representation is the representation of highest weight
$(1,\ldots ,1)$, and more generally, the power $l$ of the determinant representation 
is the irreducible representation of highest weight $(l,\ldots ,l)$.

The fact that powers of the determinant representation are of dimension one implies
that when tensored with any irreducible representation of highest weight 
$(x_1,\ldots ,x_r)$, they yield an other irreducible representation of highest weight
$(x_1+l,\ldots ,x_r+l)$.

This implies that $\bar{V}\otimes\det^{n-r}$ has highest weight of type
\begin{equation}
(n+1-k_1-r,\ldots , n+1-k_r-r),
\end{equation}
and that 
$\det^{n-r}\otimes \bar{V}_{\nu}\otimes V_{\nu}$ contains one and only one copy 
of the determinant representation $\det^{n-r}$.

We are interested in the dimension of the $GL(r,\mathbb{C})$ - Hom space 
\begin{equation}
\Hom_{GL(r,\mathbb{C})}(V_{\nu},V_{\lambda}\otimes V_{mu}):
\end{equation}
from the above facts it turns out that this dimension is exactly the
same as that of the dimension of
\begin{equation}
\Hom_{GL(r,\Cpx)}({\det}^{n-r},
{\det}^{n-r}\otimes \bar{V}_{\nu}\otimes V_{\lambda}\otimes V_{\mu}).
\end{equation}

Obviously the action by permutation of sets $I,J,K$ in $\tilde{T}_r^n$ 
corresponds to the
permutation of legs of the tensor $V_{\lambda}\otimes V_{\mu}\otimes ({\det}^{n-r}\otimes \bar{V}_{\nu})$.

The fact that the fusion rules of tensor product of groups are abelian implies that
the dimension of the Hom spaces are unchanged, so 
that $c^{(n)}(I,J,K)$ remains unchanged under permutation of indices.
\end{proof}

We now show that the reduction procedure of Remark~\ref{rem:red} preserves Littlewood--Richardson coefficients.

\begin{prop}\label{prop:redLR}
Let $(I,J,K)\in\Tt^n_r$ and suppose there are $u,v,w\in\{0,\ldots,r\}$
such that
\begin{gather}
u+v+w=r \label{eq:uvwra} \\
u=r\text{ or }i_{u+1}-i_u\ge2 \label{eq:idiffa} \\
v=r\text{ or }j_{v+1}-j_v\ge2 \label{eq:jdiffa} \\
w=r\text{ or }k_{w+1}-k_w\ge2 \label{eq:kdiffa} \\
i_u+j_v+k_w\le n-1, \label{eq:ijksumsmalla}
\end{gather}
where we set $i_0=j_0=k_0=0$.
Let $\It,\Jt,\Kt$ be as defined in~\eqref{eq:ItJtKt} and~\eqref{eq:it}--\eqref{eq:kt}.
Then $c^{(n-1)}(\It,\Jt,\Kt)=c^{(n)}(I,J,K)$.
\end{prop}
\begin{proof}
Note that $\It$, $\Jt$ and $\Kt$ are subsets of $\{1,\ldots,n-1\}$.
Let
\begin{align}
(\lambda,\mu,\nu)&=\Phi^n_r(I,J,K) \\
(\lambdat,\mut,\nut)&=\Phi^{n-1}_r(\It,\Jt,\Kt). \label{eq:lmuIJKt}
\end{align}
Then for $p\in\{1,\ldots,r\}$ we have
\begin{align}
\lambda_p&=n-r-i_p+p \\
\mu_p&=n-r-j_p+p \\
\nu_p&=k_{r+1-p}-(r+1-p).
\end{align}
Let $a=u$, $b=v$ and $c=r-w$.
Then~\eqref{eq:uvwra} gives $c=a+b$.
From~\eqref{eq:it}--\eqref{eq:kt} and~\eqref{eq:lmuIJKt},
we get
\begin{align}
\lambdat_p=
\begin{cases}
\lambda_p-1\,,&1\le p\le a, \\
\lambda_p\,,&a<p\le r
\end{cases} \label{eq:lamt} \\
\mut_p=
\begin{cases}
\mu_p-1\,,&1\le p\le b, \\
\mu_p\,,&b<p\le r
\end{cases} \\
\nut_p=
\begin{cases}
\nu_p-1\,,&1\le p\le c, \\
\nu_p\,,&c<p\le r. \label{eq:nut}
\end{cases}
\end{align}
We must show
\begin{equation}\label{eq:ctil}
c_{\lambdat,\mut}^\nut=c_{\lambda,\mu}^\nu\,.
\end{equation}

Since $c_{\lambda,\mu}^\nu=c_{\mu,\lambda}^\nu$
(see~\cite{F00} or, indeed, Lemma~\ref{lem:cIJK}), and since the statement of the lemma
is invariant under interchanging the roles of $\lambda$ and $\mu$ when we also interchange $a$ and $b$,
it follows that if we prove the lemma in some given case $a=a_0$ and $b=b_0$, then we may conclude that it also
holds in the case $a=b_0$ and $b=a_0$.

The Littlewood--Richardson coefficient $c_{\lambda,\mu}^\nu$ is equal to the
number of fillings of $\nu\backslash\lambda$ according to $\mu$, as described on page 221 of Fulton's
article~\cite{F00}.
Thus, if we let $f_\ell^k$ denote the number of times $k$ appears in the $\ell$th row, then the fillings of
$\nu\backslash\lambda$ according to $\mu$ are the choices of nonnegative integers
$(f_\ell^k)_{1\le k\le\ell}$ such that the following hold:
\begin{alignat}{2}
\lambda_\ell+\sum_{k=1}^\ell f_\ell^k&=\nu_\ell&&(1\le\ell\le r)  \label{e1} \\
\sum_{\ell=k}^r f_\ell^k&=\mu_k&&(1\le k\le r)  \label{e2} \\
\lambda_{\ell+1}+\sum_{k=1}^{p+1} f_{\ell+1}^k&\le\lambda_\ell+\sum_{k=1}^p f_\ell^k\qquad&&(0\le p<\ell< r)  \label{e3} \\
\sum_{\ell=k+1}^{p+1} f_\ell^{k+1}&\le\sum_{\ell=k}^p f_\ell^k&&(1\le k\le p< r).  \label{e4}
\end{alignat}
Indeed,~\eqref{e2} is the condition Fulton lists as~(iii), \eqref{e3} is equivalent to Fulton's~(ii),
and~\eqref{e4} is equivalent to Fulton's~(iv).

Suppose $(\ft_\ell^k)_{1\le k\le\ell\le r}$ is a filling of $\nut\backslash\lambdat$ according to $\mut$
and let
\begin{equation}\label{eq:fsub}
f^k_\ell=
\begin{cases}
\ft^k_\ell+1,&\text{if }1\le k\le b\text{ and }\ell=k+a \\
\ft^k_\ell,&\text{otherwise.}
\end{cases}
\end{equation}
We will show that the map 
\begin{equation}\label{eq:fbij}
(\ft^k_\ell)_{1\le k\le\ell\le r}\mapsto(f^k_\ell)_{1\le k\le\ell\le r}
\end{equation}
is a bijection from the set of fillings of $\nut\backslash\lambdat$ according to $\mut$
onto the set of fillings of $\nu\backslash\lambda$ according to $\mu$.
It is straightforward to show that the ``tilde'' version of
each of the equalities and inequalities~\eqref{e1}--\eqref{e4} (i.e., where
each $\lambda$, $\mu$, $\nu$ and $f^k_\ell$ is replaced by 
$\lambdat$, $\mut$, $\nut$ and $\ft^k_\ell$, respectively)
implies the ``non--tilde'' version of the same.
Here we give further information about these implications:
\begin{align*}
\text{\eqref{e1}}_{11}&\quad1\le\ell\le c \\
\text{\eqref{e1}}_{00}&\quad c<\ell\le r \\
\text{\eqref{e2}}_{11}&\quad 1\le k\le b \displaybreak[2] \\
\text{\eqref{e2}}_{00}&\quad b< k\le r  \displaybreak[2] \\
\text{\eqref{e3}}_{11}&\quad 1\le \ell\le a  \displaybreak[2] \\
\text{\eqref{e3}}_{11}&\quad a<\ell<c,\,p\ge\ell-a  \displaybreak[2] \\
\text{\eqref{e3}}_{00}&\quad a<\ell\le c,\,p<\ell-a  \displaybreak[2] \\
\text{\eqref{e3}}_{01}&\quad \ell=c<r,\,p\ge\ell-a=b  \displaybreak[2] \\
\text{\eqref{e3}}_{00}&\quad c<\ell<r  \displaybreak[2] \\
\text{\eqref{e4}}_{11}&\quad 1\le k<b,\,p\ge k+a  \displaybreak[2] \\
\text{\eqref{e4}}_{00}&\quad 1\le k\le b,\,p<k+a \\
\text{\eqref{e4}}_{01}&\quad k=b,\,p\ge k+a=c \\
\text{\eqref{e4}}_{00}&\quad b<k<r.
\end{align*}
The subscripts above indicate by how much the left-- and right--hand--sides
of the corresponding equations are incremented when changing from $\lambdat$,
$\mut$, $\nut$ and $\ft^k_\ell$ to $\lambda$, $\mu$, $\nu$ and $f^k_\ell$, respectively.
Thus, for example, the line containing \eqref{e4}$_{01}$ indicates that
when $k=b$ and $p\ge k+a$ and when we pass from
\begin{equation}
\sum_{\ell=k+1}^{p+1}\ft_\ell^{k+1}\le\sum_{\ell=k}^p\ft_\ell^k  
\end{equation}
to the inequality~\eqref{e4} by substituting $f^k_\ell$ for $\ft^k_\ell$, the value of the
right--hand--side increases
by $1$ while the value of the left--hand--side remains unchanged.
The fact that the equalities and inequalities all remain valid when making these substitutions
shows that the map~\eqref{eq:fbij} with $f^k_\ell$ defined by~\eqref{eq:fsub} is an injection from
the set of fillings of $\nut\backslash\lambdat$ according to $\mut$ into the set of fillings
of $\nu\backslash\lambda$ according to $\mu$.

To show that this map is onto is the same as showing
that whenever $(f^k_\ell)_{1\le k\le\ell\le r}$ is a filling
of $\nu\backslash\lambda$ accordiing to $\mu$, then we have
\begin{equation}\label{e6}
f^k_{a+k}>0\qquad(k\in\{1,\ldots,b\}),
\end{equation}
and if $c<r$, then
\begin{equation}\label{e7}
\lambda_{c+1}+\sum_{k=1}^{p+1}f_{c+1}^k<\lambda_c+\sum_{k=1}^pf_c^k\qquad(p\in\{b,b+1,\ldots,c-1\})
\end{equation}
and if $b>0$ and $c<r$, then
\begin{equation}\label{e8}
\sum_{\ell=b+1}^{p+1}f_\ell^{b+1}<\sum_{\ell=b}^pf_\ell^b\qquad(p\in\{c,c+1,\ldots,r-1\}),
\end{equation}
where we see~\eqref{e6} from the definition~\eqref{eq:fsub} and we see~\eqref{e7} and~\eqref{e8}
from the lines with~\eqref{e3}$_{01}$ and~\eqref{e4}$_{01}$, above.
For enough
values of $a$ and $b$ to prove the lemma,
we will use~\eqref{eq:ijksumsmalla} as well as~\eqref{e1}--\eqref{e4}
to show that the inequalities~\eqref{e6}, \eqref{e7} and~\eqref{e8} hold.

\begin{case}\label{c0} $a=b=0$.
\end{case}
Then $(\lambdat,\mut,\nut)=(\lambda,\mu,\nu)$ and~\eqref{eq:ctil} holds trivially.

\begin{case}\label{c1} $b=0$, $1\le a\le r$.
\end{case}
If $a=r$, then there is nothing to check, so assume $a<r$.
We have $a=c\in\{1,\ldots,r-1\}$ and~\eqref{eq:ijksumsmalla} becomes
\begin{equation}\label{eq:6c1}
\nu_{a+1}<\lambda_a\,,
\end{equation}
while $\ft^k_\ell=f^k_\ell$ for all $k$ and $\ell$.
It will suffice to show that~\eqref{eq:6c1} implies
\begin{equation}
\lambda_{a+1}+\sum_{k=1}^{p+1}f_{a+1}^k<\lambda_a+\sum_{k=1}^pf_a^k\qquad(p\in\{0,1,\ldots,a-1\}).
\end{equation}
But we have
\begin{equation}
\lambda_{a+1}+\sum_{k=1}^{p+1}f_{a+1}^k\le\lambda_{a+1}+\sum_{k=1}^{a+1}f_{a+1}^k=\nu_{a+1}<\lambda_a
\le\lambda_a+\sum_{k=1}^pf_a^k\,,
\end{equation}
and Case~\ref{c1} is proved.

\begin{case}\label{c2}
$1\le b<r$ and $a=r-b$.
\end{case}
Then $c=r$.
From~\eqref{eq:ijksumsmalla} we get $n-r<\lambda_a+\mu_b$, so $\nu_1<\lambda_a+\mu_b$ and therefore,
using~\eqref{e1} and~\eqref{e2}, we have
\begin{equation}\label{eq:nlmc2}
\lambda_1+f_1^1<\lambda_a+f^b_b+f^b_{b+1}+\cdots+f^b_r\,.
\end{equation}
We must only verify that~\eqref{e6} holds.
Suppose, for contradiction, that
\begin{equation}\label{eq:f0}
f^{k'}_{a+k'}=0
\end{equation}
for some $k'\in\{1,\ldots,b\}$.
Then we get
\begin{align}
f^b_b+f^b_{b+1}+\cdots+f^b_r&\le f^{b-1}_{b-1}+f^{b-1}_b+\cdots+f^{b-1}_{r-1} \label{eq:fb-1} \\
&\le f^{b-2}_{b-2}+f^{b-2}_{b-1}+\cdots+f^{b-2}_{r-2} \\
&\le\cdots\notag \\
&\le f^{k'}_{k'}+f^{k'}_{k'+1}+\cdots+f^{k'}_{a+k'} \label{eq:fk'} \\
&=f^{k'}_{k'}+f^{k'}_{k'+1}+\cdots+f^{k'}_{a+k'-1} \label{eq:fk'=} \\
&\le f^{k'-1}_{k'-1}+f^{k'-1}_{k'}+\cdots+f^{k'-1}_{a+k'-2} \label{eq:fk'-1} \\
&\le\cdots\notag \\
&\le f^1_1+f^1_2+\cdots+f^1_a\,, \label{eq:f1}
\end{align}
where in~\eqref{eq:fb-1}--\eqref{eq:fk'} we have used~\eqref{e4} with $k=b-1$ and $p=r-1$,
then with $k=b-2$ and $p=r-2$, successively to $k=k'$ and $p=r-b+k'=a+k'$,
where~\eqref{eq:fk'=} results from~\eqref{eq:f0} 
and where for~\eqref{eq:fk'-1}--\eqref{eq:f1} we used~\eqref{e4} with $k=k'-1$ and $p=a+k'-2$,
then with $k=k'-2$ and $p=a+k'-3$, successively to $k=1$ and $p=a$.
But using~\eqref{e3} with $p=1$ and, successively, $\ell=a-1,\,\ell=a-2,\ldots,\ell=1$, we have
\begin{align}
\lambda_a+f^1_a+f^1_{a-1}+\cdots+f^1_2&\le\lambda_{a-1}+f^1_{a-1}+f^1_{a-2}+\cdots+f^1_2  \label{eq:lama} \\
&\le\cdots\notag \\
&\le\lambda_2+f^1_2 \\
&\le\lambda_1\,, \label{eq:lam1}
\end{align}
which together with~\eqref{eq:fb-1}--\eqref{eq:f1} gives
\begin{equation}
\lambda_a+f^b_b+f^b_{b+1}+\cdots+f^b_r\le\lambda_1+f^1_1\,.
\end{equation}
Combining all of this with~\eqref{eq:nlmc2}, we get
\begin{equation}
\lambda_1+f^1_1<\lambda_1+f^1_1\,,
\end{equation}
a contradiction.
Thus, Case~\ref{c2} is proved.

\begin{case}\label{c3}
$1\le a\le b$ and $a+b<r$.
\end{case}
Then~\eqref{eq:ijksumsmalla} yields $\nu_1+\nu_{a+b+1}<\lambda_a+\mu_b$, or
\begin{equation}\label{eq:nlmc3}
\lambda_1+f^1_1+\lambda_{a+b+1}+f^1_{a+b+1}+f^2_{a+b+1}+\cdots+f^{a+b+1}_{a+b+1}
<\lambda_a+f^b_b+f^b_{b+1}+\cdots+f^b_r\,.  
\end{equation}
We now show that~\eqref{e6} must hold.
Supposing for contradiction that we have $f^{k'}_{k'+a}=0$ for some $k'\in\{1,\ldots,b\}$ and
arguing as we did in~\eqref{eq:fb-1}--\eqref{eq:f1}, we get
\begin{align}
f^b_b+f^b_{b+1}+\cdots+f^b_{b+a}&\le f^{b-1}_{b-1}+f^{b-1}_b+\cdots+f^{b-1}_{b+a-1} \\
&\le\cdots\notag \\
&\le f^{k'}_{k'}+f^{k'}_{k'+1}+\cdots+f^{k'}_{k'+a} \\
&=f^{k'}_{k'}+f^{k'}_{k'+1}+\cdots+f^{k'}_{k'+a-1} \\
&\le f^{k'-1}_{k'-1}+f^{k'-1}_{k'}+\cdots+f^{k'-1}_{k'+a-2} \\
&\le\cdots\notag \\
&\le f^1_1+f^1_2+\cdots+f^1_a\,.
\end{align}
Using this in~\eqref{eq:nlmc3}, we get
\begin{multline}\label{eq:lama+b+1<lama}
\lambda_1+f^1_1+\lambda_{a+b+1}+f^1_{a+b+1}+f^2_{a+b+1}+\cdots+f^{a+b+1}_{a+b+1} \\
<
\begin{aligned}[t]
&\lambda_a+f_a^1+f_{a-1}^1+\cdots+f_2^1+f_1^1 \\
&+f^b_{a+b+1}+f^b_{a+b+2}+\cdots+f^b_r\,.
\end{aligned}
\end{multline}
Using~\eqref{eq:lama}--\eqref{eq:lam1} in~\eqref{eq:lama+b+1<lama} yields
\begin{equation}\label{eq:ab1}
\lambda_{a+b+1}+f^1_{a+b+1}+f^2_{a+b+1}+\cdots+f^{a+b+1}_{a+b+1}<f^b_{a+b+1}+f^b_{a+b+2}+\cdots+f^b_r\,.
\end{equation}
Adding $\lambda_r+f^1_r+f^2_r+\cdots+f^{b-1}_r$ to the right--hand--side of~\eqref{eq:ab1}
and using~\eqref{e3} with $p=b-1$ and, successively, $\ell=r-1,\,\ell=r-2,\ldots,\ell=a+b+2$, we get
\begin{align}
\lambda_{a+b+1}+f^1_{a+b+1}&+f^2_{a+b+1}+\cdots+f^{a+b+1}_{a+b+1} \label{eq:lab1} \\
&<(\lambda_r+f^1_r+f^2_r+\cdots+f^{b-1}_r)+f^b_r+f^b_{r-1}+\cdots+f^b_{a+b+1} \\
&\le(\lambda_{r-1}+f^1_{r-1}+f^2_{r-1}\cdots+f^{b-1}_{r-1})+f^b_{r-1}+f^b_{r-2}+\cdots f^b_{a+b+1} \\
&\le\cdots\notag \\
&\le\lambda_{a+b+1}+f^1_{a+b+1}+f^2_{a+b+1}+\cdots+f^b_{a+b+1}\,. \label{eq:lab5}
\end{align}
{}From this, we get
\begin{equation}
f^{b+1}_{a+b+1}+f^{b+2}_{a+b+1}+\cdots+f^{a+b+1}_{a+b+1}<0,
\end{equation}
which is a contradiction.
Thus,~\eqref{e6} is proved.

We now show that~\eqref{e7} holds.
If it fails for some $p=p'\in\{b,b+1,\ldots,b+a-1\}$, then we must have
\begin{equation}
\lambda_{a+b+1}+f^1_{a+b+1}+f^2_{a+b+1}+\cdots+f^{p'+1}_{a+b+1}
=\lambda_{a+b}+f^1_{a+b}+f^2_{a+b}+\cdots+f^{p'}_{a+b}
\end{equation}
and then from~\eqref{eq:nlmc3} we get
\begin{equation}\label{eq:nmlc3e7}
\begin{aligned}[b]
\lambda_1+f^1_1&+(\lambda_{a+b}+f^1_{a+b}+f^2_{a+b}+\cdots+f^{p'}_{a+b}) \\
&+(f^{p'+2}_{a+b+1}+f^{p'+3}_{a+b+1}+\cdots+f^{a+b+1}_{a+b+1})
\end{aligned}
<\lambda_a+f^b_b+f^b_{b+1}+\cdots+f^b_r\,.
\end{equation}
Again using~\eqref{e4} in the familiar way, we obtain
\begin{align}
f^b_b+f^b_{b+1}+\cdots f^b_{b+a-1}&\le f^{b-1}_{b-1}+f^{b-1}_b+\cdots+f^{b-1}_{b+a-2} \\
&\le\cdots\notag \\
&\le f^1_1+f^1_2+\cdots+f^1_a\,.
\end{align}
With~\eqref{eq:nmlc3e7}, this yields
\begin{multline}\label{eq:rev2}
\lambda_1+f^1_1+(\lambda_{a+b}+f^1_{a+b}+f^2_{a+b}+\cdots+f^{p'}_{a+b})
+(f^{p'+2}_{a+b+1}+f^{p'+3}_{a+b+1}+\cdots+f^{a+b+1}_{a+b+1}) \\
<(\lambda_a+f^1_a+f^1_{a-1}+\cdots+f^1_2)+f^1_1
+(f^b_{a+b}+f^b_{a+b+1}+\cdots+f^b_r).
\end{multline}
Using~\eqref{eq:lama}--\eqref{eq:lam1} in~\eqref{eq:rev2},
we get
\begin{multline}\label{eq:rev3}
(\lambda_{a+b}+f^1_{a+b}+f^2_{a+b}+\cdots+f^{p'}_{a+b})
+(f^{p'+2}_{a+b+1}+f^{p'+3}_{a+b+1}+\cdots+f^{a+b+1}_{a+b+1}) \\
<f^b_{b+a}+f^b_{b+a+1}+\cdots+f^b_r\,.
\end{multline}
Ading $\lambda_r+f^1_r+\cdots+f^{b-1}_r$ to the right--hand--side of~\eqref{eq:rev3}
and using~\eqref{e3} with $p=b-1$ and,
successively, $\ell=r-1,\,\ell=r-2,\ldots,\ell=a+b$, we get
\begin{multline}\label{eq:rev4}
(\lambda_{a+b}+f^1_{a+b}+f^2_{a+b}+\cdots+f^{p'}_{a+b})
+(f^{p'+2}_{a+b+1}+f^{p'+3}_{a+b+1}+\cdots+f^{a+b+1}_{a+b+1}) \\
<\lambda_{a+b}+f^1_{a+b}+f^2_{a+b}+\cdots+f^b_{a+b}\,.
\end{multline}
Thus, we get
\begin{equation}
(f^{b+1}_{a+b}+f^{b+2}_{a+b}+\cdots+f^{p'}_{a+b})
+(f^{p'+2}_{a+b+1}+f^{p'+3}_{a+b+1}+\cdots+f^{a+b+1}_{a+b+1})<0,
\end{equation}
which is a contradiction,
and~\eqref{e7} is proved.

Finally, we show that~\eqref{e8} holds.
If it fails for some $p=p'\in\{a+b,a+b+1,\ldots,r-1\}$, then we have
\begin{equation}
f^{b+1}_{b+1}+f^{b+1}_{b+2}+\cdots+f^{b+1}_{p'+1}=f^b_b+f^b_{b+1}+\cdots+f^b_{p'}\,.  
\end{equation}
{}From this and~\eqref{eq:nlmc3}, we have
\begin{equation}\label{eq:nlmc38}
\begin{aligned}
\lambda_1+f^1_1&+\lambda_{a+b+1}+f^1_{a+b+1}+f^2_{a+b+1}+\cdots+f^{a+b+1}_{a+b+1} \\
&<\lambda_a
+(f^{b+1}_{b+1}+f^{b+1}_{b+2}+\cdots+f^{b+1}_{p'+1})
+(f^b_{p'+1}+f^b_{p'+2}+\cdots+f^b_r). 
\end{aligned}
\end{equation}
Arguing as before, we have 
\begin{align}
f^{b+1}_{b+1}+f^{b+1}_{b+2}+\cdots+f^{b+1}_{b+a}
&\le f^b_b+f^b_{b+1}+\cdots+f^b_{b+a-1} \\
&\le\cdots\notag \\
&\le f^1_1+f^1_2+\cdots+f^1_a\,.
\end{align}
Using this in~\eqref{eq:nlmc38}, we get
\begin{multline}
\lambda_1+f^1_1+\lambda_{a+b+1}+f^1_{a+b+1}+f^2_{a+b+1}+\cdots+f^{a+b+1}_{a+b+1} \\
<
\begin{aligned}[t]
&(\lambda_a+f_a^1+f_{a-1}^1+\cdots+f_2^1)+f^1_1 \\
&+(f^{b+1}_{a+b+1}+f^{b+1}_{a+b+2}+\cdots+f^{b+1}_{p'+1})
+(f^b_{p'+1}+f^b_{p'+2}+\cdots+f^b_r).
\end{aligned}
\end{multline}
Using this and~\eqref{eq:lama}--\eqref{eq:lam1}, we get
\begin{multline}\label{eq:rev5}
\lambda_{a+b+1}+f^1_{a+b+1}+f^2_{a+b+1}+\cdots+f^{a+b+1}_{a+b+1} \\
<
(f^{b+1}_{a+b+1}+f^{b+1}_{a+b+2}+\cdots+f^{b+1}_{p'+1})
+(f^b_{p'+1}+f^b_{p'+2}+\cdots+f^b_r).
\end{multline}
adding $\lambda_r+f^1_r+f^2+r+\cdots+f^{b-1}_r$ to the right--hand--side of~\eqref{eq:rev5} and using~\eqref{e3}
with $p=b-1$ and, successively, $\ell=r-1$, $\ell=r-2$,\ldots, $\ell=p'+1$, we get
\begin{multline}\label{eq:rev6}
\lambda_{a+b+1}+f^1_{a+b+1}+f^2_{a+b+1}+\cdots+f^{a+b+1}_{a+b+1} \\
<
(f^{b+1}_{a+b+1}+f^{b+1}_{a+b+2}+\cdots+f^{b+1}_{p'+1})
+(\lambda_{p'+1}+f^1_{p'+1}+f^2_{p'+1}+\cdots+f^b_{p'+1}) \\
=\lambda_{p'+1}+f^1_{p'+1}+f^2_{p'+1}+\cdots+f^{b+1}_{p'+1}
+(f^{b+1}_{p'}+f^{b+1}_{p'-1}+\cdots+f^{b+1}_{a+b+1}).
\end{multline}
Now using~\eqref{e3}
with $p=b$ and, successively, $\ell=p'$, $\ell=p'-1$,\ldots, $\ell=a+b+1$, we get
\begin{equation}\label{eq:rev7}
\lambda_{a+b+1}+f^1_{a+b+1}+f^2_{a+b+1}+\cdots+f^{a+b+1}_{a+b+1}
<
\lambda_{a+b+1}+f^1_{a+b+1}+f^2_{a+b+1}+\cdots+f^{b+1}_{a+b+1}\,.
\end{equation}
This implies 
\begin{equation}
  \label{eq:5}
f^{b+2}_{a+b+1}+f^{b+3}_{a+b+1}+\cdots+f^{a+b+1}_{a+b+1} <0,
\end{equation}
which is a contradiction.
Thus, Case~\ref{c3} is proved.

We have now proved enough cases so that, if we also consider also 
the cases obtained from them by interchanging $a$ and $b$,
then the lemma is proved.
\end{proof}

\begin{defi}\label{def:red}
Let $(I,J,K)\in\Tt^n_r$.
We say $(I,J,K)$ is {\em TT--reducible}, (or simply reducible)
if the method of reduction described in Remark~\ref{rem:red}
can be performed, namely, if there are $u,v,w\in\{0,\ldots,r\}$ satisfying $u+v+w=r$ and such
that~\eqref{eq:idiff}--\eqref{eq:ijksumsmall} hold, (where we take $i_0=j_0=k_0=0$).
Naturally enough, if $(I,J,K)$ is not TT--reducible, then we may say it is {\em TT--irreducible}
(or simply irreducible).
\end{defi}

\begin{lemma}\label{lem:irjrkr}
Let $n\ge r\ge2$ be integers.
If $(I,J,K)\in\Tt^n_r$ is irreducible, then $i_r=j_r=k_r=n$.
\end{lemma}
\begin{proof}
Suppose $(I,J,K)\in\Tt^n_r$ and $i_r<n$.
We will show that $(I,J,K)$ is reducible.
In view of the symmetry of $\Tt^n_r$, this will suffice to prove the lemma.

Let $u=r$ and $v=w=0$.
Then~\eqref{eq:idiff} and~\eqref{eq:ijksumsmall} both hold.
To show that $(I,J,K)$ is reducible, it will suffice to show $j_1\ge2$, for then~\eqref{eq:jdiff}
will hold and by symmetry also~\eqref{eq:kdiff} will hold.
Inspecting~\eqref{eq:Ut}, we must have $(r,1,r)\in\Ut^r_1=\Tt^r_1$.
Considering~\eqref{eq:Tt} and taking $p=1$, we must have
$i_r+j_1+k_r\ge2n+1$, so
\begin{equation}
j_1\ge(2n+1)-i_r-k_r\ge(2n+1)-(n-1)-n=2.
\end{equation}
\end{proof}

\begin{lemma}\label{lem:gaps}
Suppose $(I,J,K)\in\Tt^n_r$ satisfies $i_r=j_r=k_r=n$ and that there are $u,v,w\in\{0,1,\ldots,r\}$
such that $u+v+w=r$ and~\eqref{eq:ijksumsmall} holds, namely, $i_u+j_v+k_w\le n-1$.
Then~\eqref{eq:idiff}--\eqref{eq:kdiff} must hold.
\end{lemma}
\begin{proof}
It will suffice to show that~\eqref{eq:idiff} holds.
{}From~\eqref{eq:ijksumsmall}, we have $u\le r-1$.
\begin{case}\label{c1:gaps}
$v\ne0$ and $w=0$.  
\end{case}
Then $(\{u+1\},\{v\},\{r\})\in\Tt^r_1$, and from~\eqref{eq:Tt} we must have
\begin{equation}
i_{u+1}+j_v+k_r\ge2n+1,
\end{equation}
which yields
\begin{equation}
i_{u+1}-i_u=(i_{u+1}+j_v)-(i_u+j_v)\ge(n+1)-(n-1)=2.
\end{equation}
and~\eqref{eq:idiff} holds.

\begin{case}\label{c2:gaps}
$v\ne0$ and $w\ne0$.  
\end{case}
Then $(\{u+1,r\},\{v,r\},\{w,r\})\in\Tt^r_2$, and from~\eqref{eq:Tt} we must have
\begin{equation}
i_{u+1}+i_r+j_v+j_r+k_w+k_r\ge4n+1,
\end{equation}
which yields
\begin{equation}
i_{u+1}-i_u=(i_{u+1}+j_v+k_w)-(i_u+j_v+k_w)\ge(n+1)-(n-1)=2.
\end{equation}
and~\eqref{eq:idiff} holds.

The other case, $v=0$ and $w\ne0$, follows from symmetry considerations.
\end{proof}

The above two lemmas imply the following.
\begin{prop}\label{prop:irred}
Let $(I,J,K)\in\Tt^n_r$.
Then $(I,J,K)$ is irreducible if and only if
$u,v,w\in\{0,1,\ldots,r\}$ and $u+v+w=r$ implies
\begin{equation}\label{eq:ijksumbig}
i_u+j_v+k_w\ge n,
\end{equation}
where we set $i_0=j_0=k_0=0$.
\end{prop}

The next result describes the irreducible elements of $\Tt^n_3$ for
arbitrary $n\ge3$, which are particularly nice.
Compare this to the first part of the proof of Theorem~1 of~\cite{TT74}.

\begin{prop}\label{prop:LRmin123}
Let $1\le r\le n$ be integers and let $(I,J,K)\in\Tt^n_r$.
\begin{itemize}
\item[(i)] If $r=1$, then $(I,J,K)$ is irreducible if and only if $n=1$ and
\begin{equation}
(I,J,K)=(\{1\},\{1\},\{1\}).
\end{equation}
\item[(ii)] If $r=2$, then $(I,J,K)$ is irreducible if and only if $n=2$ and
\begin{equation}
(I,J,K)=(\{1,2\},\{1,2\},\{1,2\}).
\end{equation}
\item[(iii)] If $r=3$, then $(I,J,K)$ is irreducible if and only if
\begin{equation}\label{eq:Ttn3irr}
(I,J,K)=(\{m,m+\ell,n\},\{m,m+\ell,n\},\{m,m+\ell,n\})
\end{equation}
for some integers $\ell$ and $m$ satisfying $1\le\ell\le m$ and $2m+\ell=n$.
\end{itemize}
\end{prop}
\begin{proof}
Part~(i) follows immediately from Proposition~\ref{prop:irred}.

Part~(ii) follows easily from Proposition~\ref{prop:irred}
because if $(I,J,K)\in\Tt^n_2$ irreducible, then
\begin{equation}\label{eq:i2j2k2n}
i_2=j_2=k_2=n,
\end{equation}
while we also have
\begin{equation}
i_1+j_1+k_1=n+1
\end{equation}
from~\eqref{eq:Ut} and~\eqref{eq:i2j2k2n} and, again from Proposition~\ref{prop:irred} we get
\begin{align}
i_1+j_1&\ge n \label{eq:i1j1n} \\
i_1+k_1&\ge n \\
j_1+k_1&\ge n. \label{eq:j1k1n}
\end{align}
Adding up~\eqref{eq:i1j1n}--\eqref{eq:j1k1n}, we get
\begin{equation}
2(n+1)=2(i_1+j_1+k_1)\ge3n,
\end{equation}
so $n\le2$.

Now, for part~(iii), suppose $r=3$.
$\Tt^3_1$ consists of the triples $(\{1\},\{3\},\{3\})$ and $(\{2\},\{2\},\{3\})$ and the four
other triples obtained by permutations, while
$\Tt^3_2$ consists of the triples $(\{1,2\},\{2,3\},\{2,3\})$ and $(\{1,3\},\{1,3\},\{2,3\})$
and their permutations.
Thus, $\Tt^n_3$ is the set of triples $(I,J,K)$ satisfying
\begin{gather}
i_1+i_2+i_3+j_1+j_2+j_3+k_1+k_2+k_3=6n \label{eq:Ttn3,4} \\
i_1+i_2+j_2+j_3+k_2+k_3\ge4n+1 \\
i_2+i_3+j_1+j_2+k_2+k_3\ge4n+1 \\
i_2+i_3+j_2+j_3+k_1+k_2\ge4n+1 \\
i_1+i_3+j_1+j_3+k_2+k_3\ge4n+1 \\
i_1+i_3+j_2+j_3+k_1+k_3\ge4n+1 \\
i_2+i_3+j_1+j_3+k_1+k_3\ge4n+1 \label{eq:Ttn3,10} \\
i_1+j_3+k_3\ge2n+1 \\
i_3+j_1+k_3\ge2n+1 \\
i_3+j_3+k_1\ge2n+1 \\
i_2+j_2+k_3\ge2n+1 \label{eq:Ttn3,14} \\
i_2+j_3+k_2\ge2n+1 \\
\;i_3+j_2+k_2\ge2n+1. \label{eq:Ttn3,16}
\end{gather}
One checks that all $(I,J,K)$ of the form~\eqref{eq:Ttn3irr}
belong to $\Tt^n_3$, because~\eqref{eq:Ttn3,4}--\eqref{eq:Ttn3,16} hold,
and are irreducible, because if $u+v+w=3$,
then~\eqref{eq:ijksumbig} holds.

Let $(I,J,K)\in\Tt^n_3$ be irreducible.
Then, by Proposition~\ref{prop:irred}, 
\begin{align}
i_3=j_3=&k_3=n \label{eq:i3j3k3} \\
i_2+j_1&\ge n \label{eq:i2j1big} \\
i_2+k_1&\ge n \\
i_1+j_2&\ge n \\
i_1+k_2&\ge n \\
j_2+k_1&\ge n \\
j_1+k_2&\ge n. \label{eq:j1k2big}
\end{align}
But adding up~\eqref{eq:i2j1big}--\eqref{eq:j1k2big}, we get
\begin{equation}
2(i_1+i_2+j_1+j_2+k_1+k_2)\ge6n,  
\end{equation}
which, in light of~\eqref{eq:Ttn3,4} and~\eqref{eq:i3j3k3} must be an equality.
Thus, all of~\eqref{eq:i2j1big}--\eqref{eq:j1k2big} must be equalities,
and these imply $i_1=j_1=k_1=m$ and $i_2=j_2=k_2=m+\ell$ for some integers $m,\ell\ge1$
satisfying $2m+\ell=n$.
Using~\eqref{eq:ijksumbig} with $u=v=w=1$, we have $3m\ge n$, which implies $\ell\le m$.
\end{proof}

\begin{prop}\label{prop:LRr3min}
Let $(I,J,K)\in\Tt^n_3$ be the irreducibe Horn triple in~\eqref{eq:Ttn3irr}.
Then the Littlewood--Richardson coefficient $c^{(n)}(I,J,K)$ is equal to $\ell$.
\end{prop}
\begin{proof}
Let $(\lambda,\mu,\nu)=\Phi^n_r(I,J,K)$.
We have
\begin{equation}\label{eq:lmnvals}
\lambda=\mu=(m+\ell-2,m-1,0),\qquad\nu=(2m+\ell-3,m+\ell-2,m-1)
\end{equation}
and $c^{(n)}(I,J,K)=c_{\lambda,\mu}^\nu$ equals the number of fillings of $\nu\backslash\lambda$ according to $\mu$,
or, equivalently, the number of choices of nonnegative integers $f_1^1,f_2^1,f_2^2,f_3^1,f_3^2,f_3^3$ such that
the following hold:
\begin{align}
\lambda_1+f_1^1&=\nu_1 \label{eq:r3e1} \\
\lambda_2+f_2^1+f_2^2&=\nu_2 \label{eq:r3e2} \\
\lambda_3+f_3^1+f_3^2+f_3^3&=\nu_3 \label{eq:r3e3} \\
f_1^1+f_2^1+f_3^1&=\mu_1 \label{eq:r3e4} \\
f_2^2+f_3^2&=\mu_2 \label{eq:r3e5} \\
f_3^3&=\mu_3 \label{eq:r3e6} \\
\lambda_2+f_2^1&\le\lambda_1 \\
\lambda_3+f_3^1&\le\lambda_2 \\
\lambda_3+f_3^1+f_3^2&\le\lambda_2+f_2^1 \\
f_2^2&\le f_1^1 \\
f_2^2+f_3^2&\le f_1^1+f_2^1 \\
f_3^3&\le f_2^2\,. \label{eq:r3e12}
\end{align}
Using also the values specified in~\eqref{eq:lmnvals}, from~\eqref{eq:r3e1} and, respectively,~\eqref{eq:r3e6} we get
\begin{align}
f_1^1&=m-1 \\
f_3^3&=0.
\end{align}
A typical filling is pictured in Figure~\ref{fig:r=3fillings}.
\begin{figure}[t]
\caption{A typical filling of $\nu\backslash\lambda$ according to $\mu$.}
\label{fig:r=3fillings}
\begin{picture}(280,70)(0,0)
\put(0,60){\line(1,0){280}}
\put(0,40){\line(1,0){280}}
\put(0,20){\line(1,0){180}}
\put(0,0){\line(1,0){100}}
\put(0,0){\line(0,1){60}}
\put(60,0){\line(0,1){20}}
\put(100,0){\line(0,1){40}}
\put(120,20){\line(0,1){20}}
\put(180,20){\line(0,1){40}}
\put(280,40){\line(0,1){20}}
\put(50,47){$\lambda_1=m+\ell-2$}
\put(197,47){$f_1^1=m-1$}
\put(17,27){$\lambda_2=m-1$}
\put(105,27){$f_2^1$}
\put(142,27){$f_2^2$}
\put(25,7){$f_3^1$}
\put(72,7){$f_3^2$}
\end{picture}
\end{figure}
{}From~\eqref{eq:r3e2}--\eqref{eq:r3e5}, we get
\begin{align}
f_2^1+f_2^2&=\ell-1 \\
f_3^1+f_3^2&=m-1 \\
f_2^1+f_3^1&=\ell-1 \\
f_2^2+f_3^2&=m-1,
\end{align}
which yield 
\begin{gather}
f_3^1=f_2^2=\ell-1-f_2^1 \\
f_3^2=m-\ell+f_2^1\,.
\end{gather}
The filling is determined by the choice of $f_2^1\in\{0,1,\ldots,\ell-1\}$
and each such choice leads to all the equalities and inequalities in~\eqref{eq:r3e1}--\eqref{eq:r3e12}
being satisfied.
So we have $c_{\lambda,\mu}^\nu=\ell$.
\end{proof}

Let us say that a Horn triple $(I,J,K)\in\Tt^n_r$ is {\em LR--minimal} (or simply minimal)
if the Littlewood--Richardson
coefficient $c^{(n)}(I,J,K)$ is equal to $1$.
It is known (see Theorem~13 of~\cite{F00}) that the set of Horn inequalities corresponding
to LR--minimal Horn triples determines all of the other Horn inequalities.
For the purpose of showing that all Horn inequalities hold in all finite von Neumann algebras, it will
suffice to verify that all the LR--minimal Horn inequalities hold in all finite von Neumann algebras.

\begin{cor}\label{cor:LRminir3}
Let $n\ge3$ be an integer.
Then $\Tt^n_3$ has an element that is an
LR--minimal and irreducible Horn triple $(I,J,K)$ if and only if $n=2m+1$ is odd, and then
the unique such triple is
\begin{equation}\label{eq:LRminir3}
(I,J,K)=(\{m,m+1,n\},\{m,m+1,n\},\{m,m+1,n\}).
\end{equation}
\end{cor}

We were unable to find a nice characterization of the LR--minimal and irreducible Horn triples
in $\Tt^n_4$.
However, the complete list of such (up to permutation of $I$, $J$ and $K$) for several values of $n$
is given in Table~\ref{tab:minirred}.
These were found using the Littlewood--Richardson Calculator
package~\cite{Bu} of Anders Skovsted Buch and Maple.
\begin{table}[hb]
\caption{LR--minimal and irreducible triples in $\Tt^n_4$.}
\label{tab:minirred}
\begin{tabular}{r|l}
$n$ & $(I,J,K)$ \\[0.7ex] \hline\hline
$4$ & $(\{1,2,3,4\},\{1,2,3,4\},\{1,2,3,4\})$ \\[0.7ex] \hline
$5$ & $\emptyset$ \\[0.7ex] \hline
$6$ & (\{1,3,4,6\},\{2,3,5,6\},\{2,3,5,6\}),\;(\{2,3,4,6\},\{2,3,4,6\},\{2,3,5,6\}) \\[0.7ex] \hline
$7$ & (\{2,3,5,7\},\{2,4,5,7\},\{3,4,5,7\}),\;(\{2,3,6,7\},\{2,4,5,7\},\{2,4,5,7\}) \\[0.7ex] \hline
$8$ & (\{1,4,5,8\},\{3,4,7,8\},\{3,4,7,8\}),\;(\{2,3,5,8\},\{3,5,6,8\},\{3,5,6,8\}), \\[0.7ex]
    & (\{2,3,6,8\},\{2,5,6,8\},\{3,5,6,8\}),\;(\{2,4,5,8\},\{3,4,6,8\},\{3,4,7,8\}), \\[0.7ex]
    & (\{2,4,5,8\},\{3,4,6,8\},\{3,5,6,8\}),\;(\{3,4,5,8\},\{3,4,5,8\},\{3,4,7,8\}), \\[0.7ex]
    & (\{3,4,5,8\},\{3,4,6,8\},\{3,4,6,8\}) \\[0.7ex] \hline
$9$ & (\{2,3,6,9\},\{3,6,7,9\},\{3,6,7,9\}),\;(\{2,5,6,9\},\{3,4,7,9\},\{3,6,7,9\}), \\[0.7ex]
    & (\{2,5,6,9\},\{3,4,7,9\},\{4,5,7,9\}),\;(\{2,5,6,9\},\{3,4,8,9\},\{3,5,7,9\}), \\[0.7ex]
    & (\{3,4,6,9\},\{3,5,6,9\},\{3,6,7,9\}),\;(\{3,4,6,9\},\{3,5,6,9\},\{4,5,7,9\}), \\[0.7ex]
    & (\{3,4,6,9\},\{3,5,7,9\},\{4,5,6,9\}),\;(\{3,4,7,9\},\{3,5,6,9\},\{4,5,6,9\}), \\[0.7ex]
    & (\{3,5,6,9\},\{3,5,6,9\},\{3,4,8,9\}) \\[0.7ex] \hline
\end{tabular}
\end{table}

\section{Construction of projections}
\label{sec:Constr}

In this section, we exhibit a construct of projections which we use
in combination with results of previous sections to prove that all of the LR--minimal Horn
inequalities corresponding to triples in $\Tt^n_3$ for arbitrary $n$
must hold in all finite von Neumann algebras.

\begin{lemma}\label{lem:projconstr}
Let $\Mcal$ be a finite von Neumann algebra with normal, faithful tracial state $\tau$.
Suppose $0<\beta\le\frac25$ and $e_1,e_2,e_3\in\Mcal$ are projections with
\begin{equation}\label{eq:taueiej}
\tau(e_i)\ge\frac12+\frac\beta4,\qquad(i\in\{1,2,3\}).
\end{equation}
Then there is a projection $p\in\Mcal$ satisfying
$\tau(p)\le\frac32\beta$
and $\tau(p\wedge e_i)\ge\beta$ for all $i\in\{1,2,3\}$.
\end{lemma}
\begin{proof}
Let $q_0=e_1\wedge e_2\wedge e_3$.
\begin{case}\label{case:1}
$\tau(q_0)\ge\beta$.
\end{case}
To prove the lemma in this case, we simply let $p\le q_0$ be such that $\tau(p)=\beta$.

\smallskip
In the remaining cases, let
\begin{align}
q_1&=(e_2\wedge e_3)-q_0 \\
q_2&=(e_1\wedge e_3)-q_0 \\
q_3&=(e_1\wedge e_2)-q_0.
\end{align}
We clearly have
\begin{equation}
q_i\wedge q_j=(e_i-q_0)\wedge(e_j-q_0)=0,\qquad (i\ne j)
\end{equation}
and, using~\eqref{eq:tpvq} and~\eqref{eq:taueiej},
\begin{equation}\label{eq:tauqi}
\tau(q_0)+\tau(q_i)=\tau(e_j\wedge e_k)\ge\frac\beta2\,,
\end{equation}
where $\{i,j,k\}=\{1,2,3\}$.
We assume, without loss of generality,
\begin{equation}\label{eq:tauqord}
\tau(q_1)\ge\tau(q_2)\ge\tau(q_3).
\end{equation}

\begin{case}\label{case:7}
$\tau(q_0)\le\beta$ and
$\tau(q_1)+\tau(q_2)\ge\beta-\tau(q_0)$
\end{case}
Take projections $q_2'\le q_2$ and $q_3'\le q_3$ such that 
\begin{equation}
\tau(q_i')=\min(\tau(q_i),\frac\beta2-\frac{\tau(q_0)}2),\qquad(i\in\{2,3\})
\end{equation}
and let $q_1'\le q_1$ be such that
\begin{equation}\label{eq:tauq1'q2'}
\tau(q_1')+\tau(q_2')=\beta-\tau(q_0).
\end{equation}
Then $\tau(q_3')\le\tau(q_2')$, we have
\begin{equation}
\begin{aligned}
\tau(q_0+q_2'\vee q_3')&\le\beta \\[0.7ex]
\tau(q_0+q_1'\vee q_3')&\le\beta
\end{aligned}
\end{equation}
and we may, therefore, choose projections 
\begin{equation}
\begin{aligned}
q_4&\le e_1-(q_0+q_2'\vee q_3') \\
q_5&\le e_2-(q_0+q_1'\vee q_3')
\end{aligned}
\end{equation}
so that
\begin{equation}\label{eq:tauq4q5}
\begin{aligned}
\tau(q_4)&=\beta-\tau(q_0)-\tau(q_2')-\tau(q_3') \\
\tau(q_5)&=\beta-\tau(q_0)-\tau(q_1')-\tau(q_3').
\end{aligned}
\end{equation}
Let
\begin{equation}
p=q_0\vee q_1'\vee q_2'\vee q_3'\vee q_4\vee q_5.
\end{equation}
Then
\begin{equation}\label{eq:taup45q3'}
\begin{aligned}
\tau(p)&\le\tau(q_0)+\tau(q_1')+\tau(q_2')+\tau(q_3')+\tau(q_4)+\tau(q_5) \\[0.7ex]
&=2\beta-\tau(q_0)-\tau(q_3').
\end{aligned}
\end{equation}
If $q_3'=q_3$, then using~\eqref{eq:tauqi} we get $\tau(p)\le\frac32\beta$.
On the other hand, if $q_3'\ne q_3$, then $\tau(q_3')=\frac\beta2-\frac{\tau(q_0)}2$
and from~\eqref{eq:taup45q3'} we have
$\tau(p)\le\frac32\beta-\frac12\tau(q_0)\le\frac32\beta$.

But also 
\begin{equation}
\begin{aligned}
p\wedge e_1&\ge q_0+(q_2'\vee q_3')+q_4 \\
p\wedge e_2&\ge q_0+(q_1'\vee q_3')+q_5 \\
p\wedge e_3&\ge q_0+(q_1'\vee q_2'),
\end{aligned}
\end{equation}
which by~\eqref{eq:tauq4q5} and~\eqref{eq:tauq1'q2'}
gives $\tau(p\wedge e_i)\ge\beta$ for all $i\in\{1,2,3\}$.
This finishes the proof in Case~\ref{case:7}.

\begin{case}\label{case:8}
$\tau(q_0)\le\beta$ and $\tau(q_1)+\tau(q_2)\le\beta-\tau(q_0)$.
\end{case}
Using~\eqref{eq:tpvq}, we have
\begin{equation}
\begin{aligned}
\tau((e_1-q_0)\vee(e_2-q_0))&=\tau(e_1)+\tau(e_2)-2\tau(q_0)-\tau(q_3) \\
&\ge1+\frac\beta2-2\tau(q_0)-\tau(q_3)
\end{aligned}
\end{equation}
and, using~\eqref{eq:tpvq} again, we get
\begin{align}
\tau((e_3-(q_1\vee q_2)-q_0)&\wedge\big((e_1-q_0)\vee(e_2-q_0)\big)) \\[0.7ex]
&\ge\begin{aligned}[t]
 &(\tau(e_3)-\tau(q_1)-\tau(q_2)-\tau(q_0)) \\
 &+(1+\frac\beta2-2\tau(q_0)-\tau(q_3))-(1-\tau(q_0))
 \end{aligned} \\[0.7ex]
&\ge\frac12+\frac{3\beta}4-\sum_{i=1}^3\tau(q_i)-2\tau(q_0) \\
&\ge\frac12-\frac\beta4-\tau(q_1)-\tau(q_2)-\tau(q_0) \label{eq:penul} \\
&\ge\beta-\tau(q_1)-\tau(q_2)-\tau(q_0), \label{eq:ul}
\end{align}
where~\eqref{eq:penul} follows because
the assumptions in this case and the ordering~\eqref{eq:tauqord} imply
$\tau(q_3)\le\frac\beta2-\frac{\tau(q_0)}2\le\beta-\tau(q_0)$
and~\eqref{eq:ul} results from $\beta\le\frac25$.

Thus, we may take a projection
\begin{equation}
f\le(e_3-(q_1\vee q_2)-q_0)\wedge\big((e_1-q_0)\vee(e_2-q_0)\big)
\end{equation}
such that
\begin{equation}\label{eq:tauf}
\tau(f)=\beta-\tau(q_1)-\tau(q_2)-\tau(q_0).
\end{equation}
Let us write $E_1=E(e_1-q_0,e_2-q_0)$ and $E_2=E(e_2-q_0,e_1-q_0)$
for the idempotents defined in section~\ref{subsec:idems}.
Let $r_1=E_1^\sharp(f)$ and $r_2=E_2^\sharp(f)$.
By Lemma~\ref{lem:Ei}, we have $r_i\le e_i-q_0$ and $\tau(r_i)\le\tau(f)$ (for $i=1,2$)
and
\begin{equation}\label{eq:fr12q3}
f\le r_1\vee r_2\vee q_3.
\end{equation}
Choose any projections
\begin{equation}
\begin{aligned}
s_1&\le e_1-q_0-(r_1\vee q_2\vee q_3) \\
s_2&\le e_2-q_0-(r_2\vee q_1\vee q_3)
\end{aligned}
\end{equation}
such that
\begin{equation}\label{eq:tauq45}
\begin{aligned}
\tau(s_1)&=\beta-\tau(q_0)-\tau(r_1\vee q_2\vee q_3) \\[0.7ex]
\tau(s_2)&=\beta-\tau(q_0)-\tau(r_2\vee q_1\vee q_3).
\end{aligned}
\end{equation}
This is possible because
\begin{equation}
\tau(r_2\vee q_1\vee q_3)\le\tau(f)+\tau(q_1)+\tau(q_3)
=\beta-\tau(q_0)+\tau(q_3)-\tau(q_2)\le\beta-\tau(q_0)
\end{equation}
and, similarly, $\tau(r_1\vee q_2\vee q_3)\le\beta-\tau(q_0)$.
Let
\begin{equation}
p=q_0\vee q_1\vee q_2\vee q_3\vee r_1\vee r_2\vee s_1\vee s_2.
\end{equation}
We have
\begin{equation}
\tau(q_2\vee q_3\vee r_1\vee s_1)=\tau(s_1)+\tau(r_1\vee q_2\vee q_3)=\beta-\tau(q_0),
\end{equation}
so
\begin{equation}
\tau(q_0\vee q_1\vee q_2\vee q_3\vee r_1\vee s_1)\le\beta+\tau(q_1)
\end{equation}
and
\begin{equation}
\begin{aligned}
\tau(p)&\le\tau(q_0\vee q_1\vee q_2\vee q_3\vee r_1\vee s_1\vee s_2)+\tau(r_2)-\tau(r_2\wedge(q_1\vee q_3)) \\[0.7ex]
&\le\tau(q_0\vee q_1\vee q_2\vee q_3\vee r_1\vee s_1)+\tau(r_2)+\tau(s_2)
 -\tau(r_2\wedge(q_1\vee q_3)) \\[0.7ex]
&\le2\beta-\tau(q_0)+\tau(q_1)+\tau(r_2)-\tau(r_2\vee q_1\vee q_3)-\tau(r_2\wedge(q_1\vee q_3)) \\[0.7ex]
&=2\beta-\tau(q_0)+\tau(q_1)-\tau(q_1\vee q_3) \\[0.7ex]
&=2\beta-\tau(q_0)-\tau(q_3)\le\frac32\beta,
\end{aligned}
\end{equation}
where for the last inequality we used~\eqref{eq:tauqi}.
On the other hand, we have
\begin{equation}
\begin{aligned}
p\wedge e_1&\ge q_0\vee q_2\vee q_3\vee r_1\vee s_1=q_0+(q_2\vee q_3\vee r_1)+s_1 \\
p\wedge e_2&\ge q_0\vee q_1\vee q_3\vee r_2\vee s_2=q_0+(q_1\vee q_3\vee r_2)+s_2\,,
\end{aligned}
\end{equation}
so from~\eqref{eq:tauq45} we get $\tau(p\wedge e_i)\ge\beta$ for $i=1,2$.
Using~\eqref{eq:fr12q3}, we have
\begin{equation}
p\wedge e_3\ge q_0\vee q_1\vee q_2\vee f=q_0+(q_1\vee q_2)+f,
\end{equation}
so from~\eqref{eq:tauf} we have $\tau(p\wedge e_3)\ge\beta$.
This finishes the proof in Case~\ref{case:8}, and the lemma is proved.
\end{proof}

The above lemma applies with $\beta=\frac2n$ to give the following.
\begin{example}\label{ex:b=1}
Let $m\ge2$ be an integer and let $n=2m+1$.
Suppose $e_1,e_2,e_3$ are projections in a finite von Neumann algebra $\Mcal$
with $\tau(e_i)\ge\frac{m+1}n$, ($i\in\{1,2,3\}$).
Then there is a projection $p\in\Mcal$ with $\tau(p)\le\frac3n$ and with
$\tau(p\wedge e_i)\ge\frac2n$ ($i\in\{1,2,3\}$).
\end{example}

\begin{thm}\label{thm:Pn}
Let $1\le r\le n$ be integers and let $(I,J,K)\in\Tt^n_r$.
If either $r\in\{1,2\}$ or $r=3$ and the triple $(I,J,K)$ is LR--minimal,
then the Horn inequality corresponding to $(I,J,K)$ holds in all finite von Neumann algebras.
\end{thm}
\begin{proof}
By Proposition~\ref{prop:flagHorn}, it will suffice that each such $(I,J,K)$ has property P$_n$.
It follows from Lemma~\label{lem:redLR} that every $(I,J,K)$ can be reduced
(as in Definition~\ref{def:red})
to an irreducible triple, which will be LR--minimal if the original triple $(I,J,K)$ is LR--minimal.
By Lemma~\ref{lem:TT1}, it will, therefore, suffice to show that every irreducible
triple $(I,J,K)\in\Tt^n_r$ with $r\in\{1,2\}$ has property P$_n$, and every
irreducible and LR--minimal triple $(I,J,K)\in\Tt^n_3$ has property P$_n$.

By Proposition~\ref{prop:LRmin123},
for $r=1$ and $r=2$, we only need to verify that $(\{1\},\{1\},\{1\})\in\Tt^1_1$ has property P$_1$
and $(\{1,2\},\{1,2\},\{1,2\})\in\Tt^2_2$ has property P$_2$.
But these facts are immediate.
For $r=3$, by Corollary~\ref{cor:LRminir3}, we need only see that the triple
\begin{equation}\label{eq:mm1n}
(\{m,m+1,n\},\{m,m+1,n\},\{m,m+1,n\})  
\end{equation}
has property P$_n$, where for integers $m\ge1$ and $n=2m+1$.
When $m=1$, this is immediate from the definition.
Take $m\ge2$.
Let $e$, $f$ and $g$ be flags in a finite von Neumann algebra $\Mcal$, with specified trace $\tau$.
Then .
It follows from Lemma~\ref{lem:projconstr} (see Example~\ref{ex:b=1}) that there is a projection $p$
in $\Mcal$ such that $\tau(p)=\frac3n$, and
\begin{equation}\label{eq:pcond}
\tau(e_{\frac{m+1}n}\wedge p)\ge\frac2n\,,\qquad\tau(f_{\frac{m+1}n}\wedge p)\ge\frac2n\,,
\qquad\tau(g_{\frac{m+1}n}\wedge p)\ge\frac2n\,.
\end{equation}
\begin{equation}
\tau(e_{\frac nn}\wedge p)=\tau(f_{\frac nn}\wedge p)=\tau(g_{\frac nn}\wedge p)=\tau(p)=\frac 3n\,.
\end{equation}
Since $\tau(e_{\frac mn})=\tau(e_{\frac{m+1}n})-\frac1n$ from~\eqref{eq:tpvq} of Proposition~\ref{prop:taup} we get
\begin{align}
\tau(e_{\frac mn}\wedge p)=\tau(e_{\frac mn}\wedge(e_{\frac{m+1}n}\wedge p))
&=\tau(e_{\frac{m+1}n}\wedge p)+\tau(e_{\frac mn})-\tau((e_{\frac{m+1}n}\wedge p)\vee e_{\frac mn}) \\
&\ge\tau(e_{\frac{m+1}n}\wedge p)+\tau(e_{\frac mn})-\tau(e_{\frac{m+1}n}) \\
&=\tau(e_{\frac{m+1}n}\wedge p)-\frac1n\ge\frac1n\,,
\end{align}
and similarly
\begin{equation}\label{eq:fgwp}
\tau(f_{\frac mn}\wedge p)\ge\frac1n\qquad\tau(g_{\frac mn}\wedge p)\ge\frac1n\,.
\end{equation}
Now~\eqref{eq:pcond}--\eqref{eq:fgwp} taken together show that $p$ satisfies the requirements of
Definition~\ref{def:Pn} and the triple~\eqref{eq:mm1n} has property P$_n$.
\end{proof}

We would like to end this section with an argument that we discovered in an attempt to show that
the triple
\begin{equation}\label{eq:246}
(\{2,4,6\},\{2,4,6\},\{2,4,6\})\in\Tt^6_3
\end{equation}
has property AP$_6$.
This triple is irreducible by Proposition~\ref{prop:LRmin123} and by Proposition~\ref{prop:LRr3min},
the corresponding triple $(\lambda,\mu,\nu)$ has Littlewood--Richardson coefficient equal to $2$.
The corresponding Horn inequality,
\begin{equation}
\alpha_1+\alpha_3+\alpha_5+\beta_1+\beta_3+\beta_5\ge\gamma_2+\gamma_4+\gamma_6\,,
\end{equation}
is known to hold in all finite von Neumann algebras.
Indeed, in the $6\times 6$ matrices, by the ordering of eigenvalues, we have
\begin{equation}
\alpha_1+\alpha_3+\alpha_5+\beta_1+\beta_3+\beta_5\ge\frac12\sum_{i=1}^6(\alpha_i+\beta_i)
=\frac12\sum_{i=1}^6\gamma_i
\ge\gamma_2+\gamma_4+\gamma_6
\end{equation}
and clearly a similar argment works in finite von Neumann algebras for integrals of eigenvalue functions.
Nonetheless, it is an interesting question whether the triple~\eqref{eq:246} has property P$_6$, or at least AP$_6$.
For the latter property, given arbitrary flags $e$, $f$ and $g$ in a finite von Neumann algebra and given $\eps>0$,
we would need to find a projection $p$
such that
\begin{gather}
\tau(p)\le\frac12+\eps \\
\tau(e_{\frac26}\wedge p)\ge\frac16\,,\qquad\tau(f_{\frac26}\wedge p)\ge\frac16\,,\qquad
\tau(g_{\frac26}\wedge p)\ge\frac16 \label{eq:e26} \\
\tau(e_{\frac46}\wedge p)\ge\frac26\,,\qquad\tau(f_{\frac46}\wedge p)\ge\frac26\,,\qquad
\tau(g_{\frac46}\wedge p)\ge\frac26\,. \label{eq:g46}
\end{gather}
The following lemma
proves this, but under the added hypothesis that the projections from
the flags appearing in~\eqref{eq:e26}--\eqref{eq:g46}
be in general position.
Although we are not able to use this argument to prove that any further Horn inequalities hold
in all finite von Neumann algebras
(beyond those treated in Theorem~\ref{thm:Pn}),
we hope that the construction of projections (and in particular, the use of ``almost invariant subspaces''
in the argument) may be of interest.
\begin{lemma}
Let $e_1,e_2,e_3,f_1,f_2,f_3$ be projections in a finite von Neumann algebra $\Mcal$ with normal faithful tracial state
$\tau$,
satisfying
\begin{equation}
e_i\le f_i,\qquad\tau(e_1)=\frac13,\qquad\tau(f_i)=\frac23,\qquad(1\le i\le 3),
\end{equation}
and let $\eps>0$.
Assume further that whenever $\{i,j,k\}=\{1,2,3\}$, we have
\begin{align}
e_i\wedge f_j&=0 \label{eq:eifj0} \\
e_k\wedge(e_i\vee e_j)&=0. \label{eq:eiwejvek}
\end{align}
Then there is a projection $p\in\Mcal$ such that
\begin{alignat}{2}
\tau(p)&\le\frac12+\eps, \\
\tau(p\wedge e_i)&\ge\frac16,&\qquad(1\le i&\le 3) \\
\tau(p\wedge f_i)&\ge\frac13,&\qquad(1\le i&\le 3).
\end{alignat}
\end{lemma}
\begin{proof}
Throughout, we let $\{i,j,k\}=\{1,2,3\}$.
Let us first show
\begin{equation}\label{eq:fifj13}
\tau(f_i\wedge f_j)=\frac13\,.
\end{equation}
The inequality $\ge$ is clear from~\eqref{eq:tpvq} in Proposition~\ref{prop:taup}.
On the other hand, $e_i\wedge f_j=e_i\wedge(f_i\wedge f_j)$, so again from~\eqref{eq:tpvq}, we get
\begin{equation}
\tau(e_i\wedge f_j)\ge\tau(e_i)+\tau(f_i\wedge f_j)-\tau(f_i)=\tau(f_i\wedge f_j)-\frac13\,,
\end{equation}
so from~\eqref{eq:eifj0} we get $\le$ in~\eqref{eq:fifj13}.
By~\eqref{eq:eifj0}, we also get
\begin{equation}\label{eq:eivej}
\tau(e_i\vee e_j)=\frac23\,,
\end{equation}
which in turn yields
\begin{equation}\label{eq:fkweivej13}
\tau(f_k\wedge(e_i\vee e_j))=\frac13\,.
\end{equation}
Indeed, $\ge$ is clear from~\eqref{eq:tpvq} and~\eqref{eq:eivej}, while from~\eqref{eq:eiwejvek}
\begin{equation}
e_k\wedge(e_i\vee e_j)=e_k\wedge(f_k\wedge(e_i\vee e_j))
\end{equation}
and~\eqref{eq:tpvq} we get $\le$ in~\eqref{eq:fkweivej13}.
Let us write $E^j_i=E(e_i,e_j)$, etc.,
for the idempotents defined in Section~\ref{subsec:idems}.
We have
\begin{align}
\domproj(E^j_i)&=e_i\vee e_j-e_j \\
\kerproj(E^j_i)&=(1-e_i\vee e_j)+e_j \\
\ranproj(E^j_i)&=e_i\,.
\end{align}
Let $S^j_i=E^j_i\cdot(f_k\wedge(e_i\vee e_j))$ be the composition of operators.
We have
\begin{equation}
\kerproj(E^j_i)\wedge(f_k\wedge(e_i\vee e_j))
=(\kerproj(E^j_i)\wedge(e_i\vee e_j))\wedge f_k
=e_j\wedge f_k=0.
\end{equation}
Thus, we have
\begin{equation}
\domproj(S_i^j)=f_k\wedge(e_i\vee e_j),\qquad\ranproj(S^j_i)=e_i\,,
\end{equation}
and we have the picture in Figure~\ref{fig:spokes}, where the spokes represent projections of trace $\frac13$.
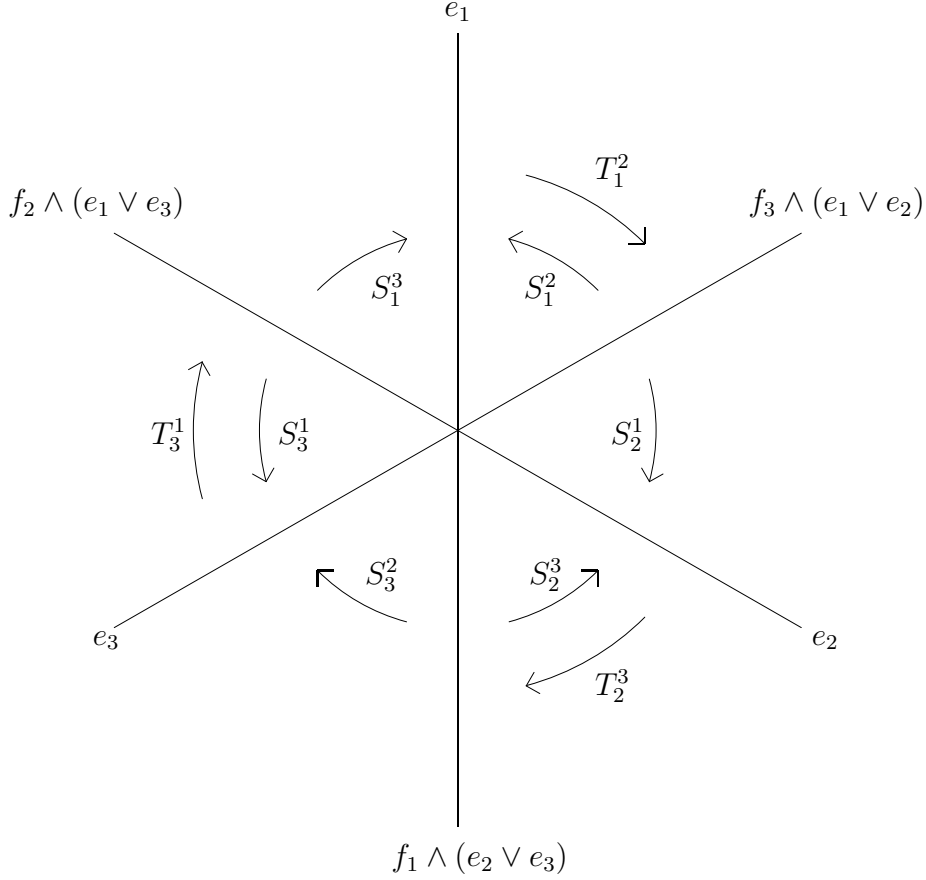
\begin{figure}[hb]
\caption{Some projections and operators.}
\label{fig:spokes}
\begin{picture}(350,345)(-175,-165)
\put(0,0){\line(87,50){130}}
\put(0,0){\line(0,1){150}}
\put(0,0){\line(-87,50){130}}
\put(0,0){\line(-87,-50){130}}
\put(0,0){\line(0,-1){150}}
\put(0,0){\line(87,-50){130}}
\put(-5,157){$e_1$}
\put(110,83){$f_3\wedge(e_1\vee e_2)$}
\put(134,-81){$e_2$}
\put(-25,-165){$f_1\wedge(e_2\vee e_3)$}
\put(-138,-81){$e_3$}
\put(-170,83){$f_2\wedge(e_1\vee e_3)$}
\put(0,0){\arc{150}{-0.262}{0.262}}
\put(0,0){\arc{150}{0.785}{1.309}}
\put(0,0){\arc{150}{1.832}{2.356}}
\put(0,0){\arc{150}{2.880}{3.404}}
\put(0,0){\arc{150}{3.927}{4.451}}
\put(0,0){\arc{150}{4.974}{5.498}}
\put(72.5,-19.4){\line(-500,866){3}}
\put(72.5,-19.4){\line(866,500){5.2}}
\put(58,-5){$S_2^1$}
\put(-53.0,-53.0){\line(0,-1){6}}
\put(-53.0,-53.0){\line(1,0){6}}
\put(-35,-58){$S_3^2$}
\put(53.0,-53.0){\line(0,-1){6}}
\put(53.0,-53.0){\line(-1,0){6}}
\put(27,-58){$S_2^3$}
\put(-72.5,-19.4){\line(500,866){3}}
\put(-72.5,-19.4){\line(-866,500){5.2}}
\put(-68,-5){$S_3^1$}
\put(-19.4,72.5){\line(-500,-866){3}}
\put(-19.4,72.5){\line(-866,500){5.2}}
\put(-33,50){$S_1^3$}
\put(19.4,72.5){\line(500,-866){3}}
\put(19.4,72.5){\line(866,500){5.2}}
\put(25,50){$S_1^2$}
\put(0,0){\arc{200}{0.785}{1.309}}
\put(0,0){\arc{200}{2.880}{3.404}}
\put(0,0){\arc{200}{4.974}{5.498}}
\put(-96.7,25.9){\line(-866,-500){5.2}}
\put(-96.7,25.9){\line(500,-866){3}}
\put(-116,-5){$T_3^1$}
\put(70.7,70.7){\line(-1,0){6}}
\put(70.7,70.7){\line(0,1){6}}
\put(52,95){$T_1^2$}
\put(25.9,-96.7){\line(500,866){3}}
\put(25.9,-96.7){\line(866,-500){5.2}}
\put(52,-100){$T_2^3$}
\end{picture}
\end{figure}

Consider the operator
\begin{equation}
X=S_1^3T_3^1S_3^2T_2^3S_2^1T_1^2\,.
\end{equation}
Then $X$ goes once around the wheel in Figure~\ref{fig:spokes}.
Since we have
\begin{gather}
\domproj(T^2_1)=e_1 \\
\ranproj(T^2_1)=f_3\wedge(e_1\vee e_2)=\domproj(S_2^1) \\
\ranproj(S^1_2)=e_2=\domproj(T_2^3) \\
\ranproj(T^3_2)=f_1\wedge(e_2\vee e_3)=\domproj(S_3^2) \\
\ranproj(S^2_3)=e_3=\domproj(T_3^1) \\
\ranproj(T^1_3)=f_2\wedge(e_1\vee e_3)=\domproj(S_1^3) \\
\ranproj(S^3_1)=e_1\,,
\end{gather}
we see
\begin{equation}
\domproj(X)=e_1=\ranproj(X).
\end{equation}
By Proposition~\ref{prop:ais}, there is a projection $q_1\le e_1$ such that
\begin{equation}
\tau(q_1)=\frac16\,,\qquad\tau(q_1\vee X^\sharp(q_1))\le\frac16+\eps.
\end{equation}
Let
\begin{align}
r_3&=(T_1^2)^\sharp(q_1) \\
q_2&=(S_2^1)^\sharp(r_3) \\
r_1&=(T_2^3)^\sharp(q_2) \\
q_3&=(S_3^2)^\sharp(r_1) \\
r_2&=(T_3^1)^\sharp(q_3).
\end{align}
Thus $(S_1^3)^\sharp(r_2)=X^\sharp(q_1)$ and $\tau(q_i)=\tau(r_i)=\frac16$, ($i=1,2,3$).
Let
\begin{equation}
p=q_1\vee q_2\vee q_3\vee X^\sharp(q_1).
\end{equation}
Then
\begin{equation}
\tau(p)\le\tau(q_1\vee X^\sharp(q_1))+\tau(q_2)+\tau(q_3)\le\frac12+\eps
\end{equation}
and $p\wedge e_i\ge q_i$, so
\begin{equation}
\tau(p\wedge e_i)\ge\frac16\,,\quad(i=1,2,3).
\end{equation}
On the other hand, we have
\begin{align}
(E_1^2)^\sharp(r_3)=(S_1^2)^\sharp(r_3)&=q_1 \\
(E_2^1)^\sharp(r_3)=(S_2^1)^\sharp(r_3)&=q_2
\end{align}
and from Lemma~\ref{lem:Ei}, we get $r_3\le q_1\vee q_2$.
Similarly, we get $r_1\le q_2\vee q_3$ and $r_2\le X^\sharp(q_1)\vee q_3$.
Thus, for every $k\in\{1,2,3\}$, we have $r_k\le p$ and
$f_k\wedge p\ge q_k\vee r_k$.
Since $q_k\wedge r_k\le e_k\wedge(e_i\vee e_j)=0$, where $\{i,j,k\}=\{1,2,3\}$, 
we have
\begin{equation}
\tau(f_k\wedge p)\ge\tau(q_k\vee r_k)=\tau(q_k)+\tau(r_k)=\frac13
\end{equation}
and the lemma is proved.
\end{proof}

\section{Possibilities for construction of a non--embeddable example}
\label{sec:constr}

This section is speculative and can be skipped without compromising understanding of the rest of the paper.

Suppose you knew, (say, you met a time traveler from the future),
that Connes' embedding problem has a negative answer and, even more, that the Horn inequality
associated to a triple $(I,J,K)\in T^n_r$
fails to hold in some finite von Neumann algebra.
How could you find and describe a finite von Neumann algebra where this Horn inequality fails?
In this section, we describe an approach, though it is not one that would be guaranteed to work.
We actually attempted to carry out this approach, without success, at the beginning of our work with Horn
inequalities in finite von Neumann algebras.
We did not benefit from an oracle of any sort, and we chose a Horn inequality (to try to violate in a finite
von Neumann algebra) by simple guessing.
(The particular one that we chose is, in fact, now known to hold in all finite von Neumann algebras, by results of this paper.)

We seek operators $a$ and $b$ whose distributions are, respectively,
\begin{align}
\mu_a&=\sum_{i=1}^n\delta_{\alpha_i} \\
\mu_b&=\sum_{j=1}^n\delta_{\beta_j}
\end{align}
and we postulate that $a+b$ has distribution
\begin{equation}
\mu_{a+b}=\sum_{k=1}^n\delta_{\gamma_k}\,,  
\end{equation}
for some real numbers
\begin{gather}
\alpha_1\ge\alpha_2\ge\cdots\ge\alpha_n \label{eq:alpha1n} \\  
\beta_1\ge\beta_2\ge\cdots\ge\beta_n \\  
\gamma_1\ge\gamma_2\ge\cdots\ge\gamma_n\,, \label{eq:gamma1n}
\end{gather}
where the trace equality
\begin{equation}\label{eq:tr=}
\sum_{i=1}^n\alpha_i+\sum_{j=1}^n\beta_j=\sum_{k=1}^n\gamma_k 
\end{equation}
holds and Horn's inequality~\eqref{eq:HIJK} fails.
After rescaling, we may suppose
\begin{equation}\label{eq:HIJKf}
1+\sum_{i\in I}\alpha_i+\sum_{j\in J}\beta_j=\sum_{k\in K}\gamma_k
\end{equation}
In fact, pick some specific values of $\alpha_1,\ldots,\alpha_n$, $\beta_1,\ldots,\beta_n$
and $\gamma_1,\ldots,\gamma_n$ such that \eqref{eq:alpha1n}--\eqref{eq:gamma1n}, \eqref{eq:tr=}
and~\eqref{eq:HIJKf} all hold.
Finding a finite von Neumann algebra in which such $a$ and $b$ can be found is equivalent
to finding a positive trace $\tau$ on the algebra $\Cpx\langle X,Y\rangle$ of polynomials in noncommuting
variables $X$ and $Y$ such that $\tau(1)=1$, and for all $k\in\Nats$, we have
\begin{align}
\tau(X^k)&=\frac1n\sum_{i=1}^n\alpha_i^k \label{eq:tauXk} \\
\tau(Y^k)&=\frac1n\sum_{i=1}^n\beta_i^k \label{eq:tauYk} \\
\tau((X+Y)^k)&=\frac1n\sum_{i=1}^n\gamma_i^k\,. \label{eq:tauX+Yk}
\end{align}
Indeed, such a trace will give rise, via the Gelfand--Naimark--Segal construction, to a Hilbert space and
a representation of $\Cpx\langle X,Y\rangle$ whose closure in the strong operator topology is
the desired finite von Neumann algebra, with $a$ and $b$ being the images of $X$ and $Y$ under
the representation.

For such a trace $\tau$, the moments $\tau(X^k)$ and $\tau(Y^k)$ are, of course, specified
by~\eqref{eq:tauXk} and~\eqref{eq:tauYk} above.
It remains to choose values for mixed moments.
In light of the trace property, this amounts to choosing values for 
all expressions of the form
\begin{equation}\label{eq:tauXY}
\tau(X^{p_1}Y^{q_1}\cdots X^{p_\ell}Y^{q_\ell})
\end{equation}
for positive integers $\ell$ and $p_1,q_1,\ldots,p_\ell,q_\ell$.
Of course, the trace condition implies that the value of~\eqref{eq:tauXY} is unchanged by cyclically
permuting the $\ell$ pairs $(p_1,q_1),\ldots,(p_\ell,q_\ell)$.
For convenience, let us say that the expresssion~\eqref{eq:tauXY} is in {\em canonical form} if
$p_1=\max_{1\le j\le\ell}p_j$ and $q_1=\max_{\{j\mid p_j=p_1 \}}q_j$ and $p_2=\max_{\{j\mid p_j=p_1,\,q_j=q_1\}}p_{j+1}$,
{\em etc.},
and we choose values of the mixed moments~\eqref{eq:tauXY} that are in canonical form.

Some linear relations between these moments are implied by the predetermined values found in~\eqref{eq:tauX+Yk}.
For example, taking $k=2,3,4$, we get
\begin{gather}
2\tau(XY)=\tau((X+Y)^2)-\tau(X^2)-\tau(Y^2) \\
3(\tau(X^2Y)+\tau(XY^2))=\tau((X+Y)^3)-\tau(X^3)-\tau(Y^3) \\
4(\tau(X^3Y)+\tau(XY^3))+2\tau(XYXY)=\tau((X+Y)^4)-\tau(X^4)-\tau(Y^4).
\end{gather}
Finally, the positivity of $\tau$ is equivalent to the positive semidefinitenesss of every matrix
of the form 
\begin{equation}
\big(\tau(w_i^*w_j))_{1\le i,j\le n}\,,  
\end{equation}
for every finite list $(w_1,\ldots,w_n)$ of distinct words in the free semigroup generated by $X$ and $Y$,
where $w_i^*$ is the word $w_i$ taken in reverse order.

\bibliographystyle{plain}

\end{document}